\documentclass[a4paper,10pt
]{article}%

\usepackage{filecontents}

\usepackage[dvipsnames]{xcolor}

\usepackage{xr-hyper} 

\usepackage[pagebackref, colorlinks, citecolor=PineGreen, linkcolor=PineGreen]{hyperref}
\hypersetup{
  final,
  pdftitle={Rigidification of dendroidal $\infty$-operads}
  pdfauthor={Bonventre, P. and Pereira, L. A.},
  linktoc=page,
  pdfpagemode=UseNone,
  pdfpagelayout=SinglePage,
}

\usepackage{amsmath, amsthm}

\usepackage[latin1]{inputenc}%
\usepackage{MnSymbol}

\usepackage[normalem]{ulem}
\usepackage{bbm}

\DeclareMathAlphabet\mathbb{U}{msb}{m}{n}
\usepackage{upgreek}
\usepackage{mathrsfs}

\usepackage[inline,shortlabels]{enumitem}%
\setenumerate{label=(\roman*)}

\usepackage[final]{microtype}
\usepackage{relsize}
\usepackage{geometry}

\usepackage{tikz}%
\usetikzlibrary{matrix,arrows,decorations.pathmorphing,
cd,patterns,calc}
\tikzset{%
  dummy/.style    = {circle,draw,inner sep=0pt,minimum size=2mm}%
}%

\usetikzlibrary[decorations.pathreplacing]

\makeatletter

\def\@testdef #1#2#3{%
  \def\reserved@a{#3}\expandafter \ifx \csname #1@#2\endcsname
  \reserved@a  \else
  \typeout{^^Jlabel #2 changed:^^J%
    \meaning\reserved@a^^J%
    \expandafter\meaning\csname #1@#2\endcsname^^J}%
  \@tempswatrue \fi}

\makeatother

\numberwithin{equation}{section} 
\numberwithin{figure}{section}

\usepackage{mathtools}
\mathtoolsset{showonlyrefs,showmanualtags} 

\newtheorem{theorem}[equation]{Theorem}%
\newtheorem*{theorem*}{Theorem}%
\newtheorem{lemma}[equation]{Lemma}%
\newtheorem{proposition}[equation]{Proposition}%
\newtheorem{corollary}[equation]{Corollary}%
\newtheorem*{conjecture*}{Conjecture}%
\newtheorem{innercustomgeneric}{\customgenericname}
\providecommand{\customgenericname}{}
\newcommand{\newcustomtheorem}[2]{%
  \newenvironment{#1}[1]
  {%
   \renewcommand\customgenericname{#2}%
   \renewcommand\theinnercustomgeneric{##1}%
   \innercustomgeneric
  }
  {\endinnercustomgeneric}
}

\newcustomtheorem{customthm}{Theorem}
\newcustomtheorem{customprop}{Proposition}
\newcustomtheorem{customcor}{Corollary}

\theoremstyle{definition} 
\newtheorem{definition}[equation]{Definition}%
\newtheorem*{definition*}{Definition}%
\newtheorem{example}[equation]{Example}%
\newtheorem{remark}[equation]{Remark}%
\newtheorem{notation}[equation]{Notation}%
\usepackage{harpoon}
\newcommand{\vect}[1]{\text{\overrightharp{\ensuremath{#1}}}}

\newcommand{\Set}{\ensuremath{\mathsf{Set}}}

\newcommand{\sSet}{\ensuremath{\mathsf{sSet}}}%

\newcommand{\Op}{\mathsf{Op}}%
\newcommand{\sOp}{\ensuremath{\mathsf{sOp}}}%
\newcommand{\dSet}{\mathsf{dSet}}

\newcommand{\Fun}{\mathsf{Fun}}

\DeclareMathOperator{\colim}{colim}%
\renewcommand{\O}{\ensuremath{\mathcal O}}

\title{Rigidification of dendroidal infinity-operads}

\author{Peter Bonventre, Lu\'is A. Pereira}%

\date{\today}

\begin{document}

\maketitle

\begin{abstract}
	We give an explicit description
	of the rigidification of an $\infty$-operad
	as a simplicial operad.
	This description is based on the notion of 
	dendroidal necklace, 
	extending work of 
	Dugger and Spivak
	from the categorical context to 
	the operadic context,
	although with a different framework,
	which relates constructions involving necklaces 
	to a standard factorization
	of maps in the category of trees.
\end{abstract}

\tableofcontents

\section{Introduction}

The notion of \emph{$\infty$-operad} is a generalization of the notion of \emph{(colored) operad}
(also sometimes called \emph{multicategories}),
introduced by Moerdijk-Weiss \cite{MW07},
where composition of operations is only defined 
``up to a contractible space of choices'',
in the same way that quasi-categories generalize categories.
Moreover, just as quasi-categories are defined as those simplicial sets 
$X \in \mathsf{sSet} = \mathsf{Set}^{\Delta^{op}}$
(for $\Delta$ the simplicial category)
satisfying a lifting condition against inner horn inclusions,
so too are $\infty$-operads defined as those \emph{dendroidal sets}
$X \in \mathsf{dSet} = \mathsf{Set}^{\Omega^{op}}$
(for $\Omega$ the category of trees)
satisfying a lifting condition against 
dendroidal inner horn inclusions \cite[\S 2.1]{CM11}.

There are two main procedures for converting a presheaf
$X \in \mathsf{dSet}$
into a (strict) operad,
given by the left adjoints
$W_!$, $\tau$
in two adjunctions as below,
where $\Op$ (resp. $\sOp$) 
denotes operads of sets (resp. of simplicial sets).
\begin{equation}\label{SOPDSET_EQ}
	W_! \colon \dSet \rightleftarrows \sOp \colon h c N
	\qquad
	\tau \colon \dSet \rightleftarrows \Op \colon N
\end{equation}
Before recalling how these adjunctions are defined,
we discuss their importance. 
First, in the $(\tau,N)$-adjunction,
the right adjoint, the \emph{nerve} $N$,
is a fully faithful inclusion whose image consists of 
(certain) $\infty$-operads, 
cf. Remark \ref{SSCOTHER REM},
thus making precise the idea that 
$\infty$-operads generalize operads.
On the other hand, the $(W_!,hcN)$-adjunction
is central for the homotopy theory of $\infty$-operads,
as it was shown to be a Quillen equivalence \cite{CM13b}
between the model structure on $\dSet$ 
(with fibrant objects the $\infty$-operads)
and the canonical model structure on $\sOp$.
Moreover, in \cite{BP_TAS} the authors established an equivariant
version of the Quillen equivalence in \cite{CM13b},
modeling the homotopy theory of 
\emph{equivariant operads with norm maps}.
In particular, 
our work here plays a minor but necessary role in the proofs
in \cite{BP_TAS},
cf. \cite[Lemma \ref{TAS-WLEFTQPUSH LEM}]{BP_TAS},
by giving an explicit description\footnote{It is worth noting that the description of these specific operads is well known,
yet the extant references we are aware of seem to leave
this description as an exercise to the reader.}
of the simplicial operads
$W_!(\partial \Omega[T]),W_!(\Lambda^E[T])$,
cf. Examples \ref{WPARTIALT_EX},\ref{WPARTIALT2_EX}.

Common to both adjunctions in \eqref{SOPDSET_EQ}
is that the right adjoints
$hcN$, $N$ are straightforward to describe
(cf. \eqref{TWONER EQ}), 
while the left adjoints $W_!$, $\tau$ are more mysterious,
as they involve colimits in operads (cf. \eqref{WXDEF_EQ}).
The main goal of this paper, 
which is an offshoot of our work in \cite{BP_TAS},
is to give an explicit description of $W_!$,
cf. Theorem \ref{THMB},
generalizing work of Dugger and Spivak \cite{DS11}
in the context of quasi-categories.
Additionally, a variation of our main constructions
gives a description of the simpler functor $\tau$
(cf. Remark \ref{TAUFUNEX REM}).

We now recall the definitions of the functors in
\eqref{SOPDSET_EQ}.
First, each tree $T \in \Omega$
has an associated colored operad
$\Omega(T) \in \mathsf{Op}$
with colors the edges of the tree and 
operations generated by the nodes
(\cite[\S 3]{MW07}; see also Example \ref{OMTSEG EX}).
Moreover, there is a ``fattened'' replacement 
$W(T) \in \sOp$ for $\Omega(T)$,
which can be built \cite[Rem. 7.3]{MW09}
as the Boardman-Vogt construction on $\Omega(T)$
(though here we use a novel description of $W(T)$,
cf. Proposition \ref{PROPA PROP}),
which replaces the non-empty mapping sets of 
$\Omega(T)$, which are all singletons $\**$,
with larger \emph{contractible spaces}.
The functors $hcN$ and $N$,
which are called, respectively,
the \emph{homotopy coherent nerve}
and the \emph{nerve},
are then given by (where $\O$ is in $\sOp$ or $\Op$ as appropriate)
\begin{equation}\label{TWONER EQ}
	hcN \O(T) = \sOp(W(T),\O),
	\qquad
	N \O(T) = \Op(\Omega(T),\O).
\end{equation}
Loosely speaking, $hcN$ can thus be regarded as a variant of $N$
obtained by replacing the notion of strict equality with that of homotopy.
Writing 
$\Omega[T]\in \mathsf{dSet} = \mathsf{Set}^{\Omega^{op}}$
for the representable functor
$\Omega[T](-) = \mathsf{dSet}(-,T)$
associated to $T \in \Omega$,
by abstract nonsense one then has the formulas
\begin{equation}\label{WXDEF_EQ}
	W_!X = \colim_{\Omega[T] \to X} W(T),
\qquad
	\tau = \colim_{\Omega[T] \to X} \Omega(T).
\end{equation}
However, as previously noted, 
the colimits in \eqref{WXDEF_EQ}
take place in $\sOp$, $\Op$,
making these formulas rather opaque.
Just as in the work in \cite{DS11} in the categorical context,
the key to obtaining explicit formulas for
$W_!$, $\tau$ will be the notion of (dendroidal) necklace,
which we now introduce
(the reason why necklaces are useful in this process is explained following \eqref{KEYSTRAT EQ}).

In the work of Dugger and Spivak in the categorical context
\cite{DS11},
a \emph{necklace} is a simplicial set of the form
$\Delta^{n_1} \vee 
\Delta^{n_2} \vee \cdots \vee
\Delta^{n_k}$
where each $\Delta^{n_i}$ is glued along its terminal vertex to the initial vertex of $\Delta^{n_{i+1}}$.
Moreover, we demand $n_i > 0$
except for the necklace $\Delta^0$ consisting of a single point.
On a terminological note,
the initial and terminal vertices of the 
$\Delta^{n_i}$ are called the \emph{joints} of the necklace,
while the $\Delta^{n_i}$ with $n_i>0$ are called beads\footnote{In particular, 
we consider the exceptional necklace $\Delta^0$ to
have no beads.  
This differs slightly from the convention in \cite{DS11}, 
which regards $\Delta^0$ as a bead of the necklace $\Delta^0$. 
This ultimately makes little difference,
but in our convention beads are always in bijection with 
vertices of the \emph{tree of joints}, cf. Figure \ref{FIGURE},
as discussed below.}.
Since the $\Delta^n$ are simply the representable presheaves in 
$\mathsf{sSet}$,
their role in the operadic context is naturally played by the 
representable presheaves $\Omega[T]$ of $\mathsf{dSet}$
for $T$ a tree.
However, formulating the notion of necklace in the dendroidal context requires some care.
This is because, while each $\Omega[T]$ does have a terminal vertex, 
corresponding to the root of $T$, 
it in general has 
\emph{multiple} initial vertices, corresponding to the leaves
of $T$.
As such, when specifying a dendroidal necklace
one must also specify the leaves along which to glue.
As an example, the tree arrangement of the trees
$T_1,T_2,T_3,T_4,T_5$ on the left
in Figure \ref{FIGURE}
\begin{figure}[ht]
\[
\begin{tikzpicture}[grow=up,auto,level distance=2.1em,
every node/.style = {font=\footnotesize,inner sep=2pt},
dummy/.style={circle,draw,inner sep=0pt,minimum size=1.375mm}]
\begin{scope}[xshift=-2em]
\begin{scope}
\tikzstyle{level 2}=[sibling distance=2.25em]%
\tikzstyle{level 3}=[sibling distance=1.25em]%
\node at (-0.5,1) [font = \normalsize] {$T_1$}
child{node [dummy] {}
	child{
	edge from parent node [near end,swap]{$l_3$}}
	child{node [dummy] {}
		child{
		edge from parent node [near end,swap]{$l_2$}}
		child{
		edge from parent node [near end]{$l_1$}}
	}
	edge from parent node {$a$}};
\end{scope}
\node at (1,1) [font = \normalsize] {$T_2$}
	child{node [dummy] {}
		child{node [dummy] {}
			child{node [dummy] {}}
	}
	edge from parent node [swap] {$b$}};
\begin{scope}
\tikzstyle{level 2}=[sibling distance=0.95em]%
\node at (2.05,3) [font = \normalsize] {$T_3$}
child{node [dummy] {}
	child{node [dummy] {}}
	child{node [dummy] {}}
	edge from parent node {$c$}};
\end{scope}
\begin{scope}
\tikzstyle{level 2}=[sibling distance=2.5em]%
\node at (2.5,1) [font = \normalsize] {$T_4$}
	child{node [dummy] {}
		child{node[dummy] {}
			child{node[dummy] {}}
		}
		child{
		edge from parent node {$c$}}
	edge from parent node [swap] {$d$}};
\end{scope}
\begin{scope}
\tikzstyle{level 2}=[sibling distance=4.25em,level distance=1.8em]%
\node at (1,-1) [font = \normalsize] {$T_5$}
child{node [dummy] {}
	child[sibling distance = 4em]{edge from parent node [swap] {$d$} }
	child[sibling distance = 3.5em]{edge from parent node [swap,near end] {$b$} }
	child[sibling distance = 4em]{ edge from parent node {$\phantom{d}a$} }
	edge from parent node [swap] {$r$}};
\end{scope}
\end{scope}
\begin{scope}[yshift=-0.5em,xshift=2.4em]
\begin{scope}[level distance=2.3em]
\tikzstyle{level 2}=[sibling distance=4em]%
\tikzstyle{level 3}=[sibling distance=2.75em]%
\tikzstyle{level 4}=[sibling distance=1.25em]%
\tikzstyle{level 5}=[sibling distance=0.875em]%
\node at (5.375,0) [font = \normalsize] {$T$}
	child{node [dummy] {}
		child[sibling distance = 4.25em]{node [dummy] {}
			child{node [dummy] {}
				child{node [dummy] {}}
			}
			child{node [dummy] {}
				child{node [dummy] {}}
				child{node [dummy] {}}
			edge from parent node [near end] {$c$}}
		edge from parent node [swap] {$d$}}
		child[sibling distance =3.5em]{node [dummy] {}
			child{node [dummy] {}
				child{node [dummy] {}}
			}
		edge from parent node [swap,near end] {$b$}}
	child[sibling distance =4em]{node [dummy] {}
		child{
		edge from parent node [near end,swap]{$l_3$}}
		child{node [dummy] {}
			child{
			edge from parent node [near end,swap]{$l_2$}}
			child{
			edge from parent node [near end]{$l_1$}}
		}
		edge from parent node {$\phantom{d}a$}}
	edge from parent node [swap] {$r$}};
\end{scope}
\begin{scope}[level distance=2.3em]
\tikzstyle{level 2}=[sibling distance=2.6em]%
\tikzstyle{level 3}=[sibling distance=1.45em]%
\node at (10.5,0.3) [font = \normalsize] {$J$}
	child{node [dummy] {}
		child{node [dummy] {}
			child{node [dummy] {}
			edge from parent node [swap] {$c$}}	
		edge from parent node [swap] {$d$}}
		child{node [dummy] {}
		edge from parent node [near end,swap] {$b$}}
		child{node [dummy] {}
			child{
			edge from parent node [near end,swap]{$l_3$}}
			child[level distance=2.75em]{
			edge from parent node [swap,near end]{$l_2$}}
			child{
			edge from parent node [near end]{$l_1$}}
		edge from parent node {$\phantom{d}a$}}
	edge from parent node [swap] {$r$}};
\end{scope}
\draw [->] (9.1,1) -- node {$\mathfrak{n}$} (7.6,1);
\end{scope}
\end{tikzpicture}
\]
\caption{Encoding a dendroidal necklace}
\label{FIGURE}
\end{figure}
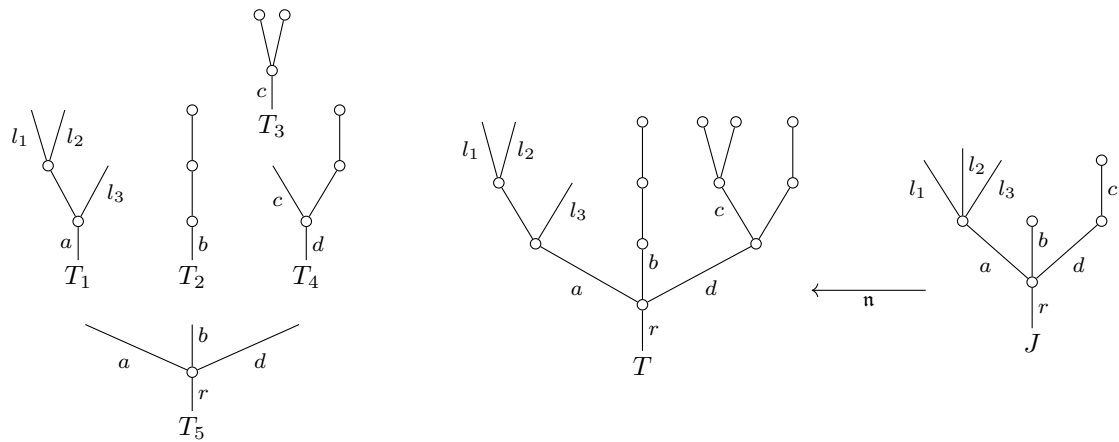
gives rise to a dendroidal necklace
(where $\amalg_e$ denotes gluing along the edge $e$)
\begin{equation}\label{DENDNECBAD EQ}
\Omega[T_1] \amalg_a 
\left(
\Omega[T_2] \amalg_b
\left(
\Omega[T_3] \amalg_c
\Omega[T_4] \amalg_d
\Omega[T_5]
\right)
\right)
\end{equation}
However, in practice \eqref{DENDNECBAD EQ} is rather awkward
to work with due to the need to include brackets, 
as well as the existence of distinct bracketing orders.
To address this, we will prefer a different presentation of dendroidal necklaces.
First, note that by gluing (also known as \textit{grafting}) the trees $T_i$ in 
Figure \ref{FIGURE} as suggested by the arrangement, 
one obtains the tree $T$ therein.
In addition, the tree $J$ encodes the \emph{arrangement}
of the $T_i$ itself. 
To make this more precise,
note that $J$ can be obtained by replacing each $T_i$ 
in the left arrangement with the corolla (i.e. tree with a single vertex) with the same number of leaves,
and then performing the grafting.
Moreover, this procedure gives rise to the indicated map
$\mathfrak{n} \colon J \to T$
in Figure \ref{FIGURE},
which completely encodes the left arrangement of the $T_i$: 
inner edges of $J$ encode the gluing edges;
the vertices of $J$ are in natural bijection with the set of the $T_i$; the $T_i$ themselves are the (outer) subtrees of $T$ whose outer edges (i.e. leaves and root)
are the image under $\mathfrak{n}$
of the corresponding vertex of $J$.
As such, we will regard such maps 
$\mathfrak{n} \colon J \to T$ themselves 
as our description of a dendroidal necklace,
cf. Definition \ref{NECKLACE_DEF}
(more precisely, necklaces are then the 
\emph{planar inner face maps} in $\Omega$).
We note that, should all $T_i$ be linear trees, so that the dendroidal necklace is one of the simplicial necklaces
$\Delta^{n_1} \vee \cdots \vee \Delta^{n_k}$,
the edges of $J$ (which is then also linear) correspond to the joints of the necklace. 
As such, we refer to the tree $J$ in a necklace as the 
\emph{tree of joints}. Similarly, we call the $T_i$ the \emph{beads} of the necklace, where we require that beads $T_i$ always have at least one vertex (generalizing the $n_i>0$ requirement in the simplicial context).

We end this introduction by observing that our presentation of necklaces as maps $\mathfrak{n} \colon J \to T$
foreshadows our approach throughout the paper.
More precisely, all our main constructions and proofs
(e.g. Definition \ref{NWTNS DEF} and
Proposition \ref{NWTNS_NAT_PROP})
are formal consequences of a standard factorization of maps in the category $\Omega$ of trees, cf. Proposition \ref{TREEFACT_PROP}.
Notably, this is rather different from the approach in \cite{DS11},
despite our approach broadly paralleling theirs,
and we believe this more formal approach is of intrinsic value,
as it may prove easier to generalize to other contexts.

\subsection{Main result}
\label{MAINRESULT_SEC}

As noted following \eqref{SOPDSET_EQ},
the nerve $N\colon \mathsf{Op} \to \mathsf{dSet}$
is fully faithful.
Moreover, its (essential) image
can be characterized as those 
$X \in \mathsf{dSet}$ satisfying a strict Segal condition,
recalled in \eqref{SCV1 EQ},\eqref{SCV2 EQ}.
As such, we will throughout make use of the following trick:
rather than describe an operad $\O \in \mathsf{Op}$,
we directly build its nerve 
$N\O \in \mathsf{dSet}$ as a presheaf,
then check that the described $N\O$ satisfies the required strict Segal condition.
The advantage of this trick is that it provides rather compact descriptions of the main operads we care about 
(cf. Definition \ref{NWTNS DEF}).
For instance, 
the operad $\Omega(T)$ appearing in 
\eqref{TWONER EQ},\eqref{WXDEF_EQ}
is characterized by the identification $N\Omega(T) =\Omega[T]$
(see also Example \ref{OMTSEG EX}), 
where we recall that $\Omega[T]\in \mathsf{dSet}$
is simply the representable $\Omega[T](-) = \Omega(-,T)$.

In addition, recalling that $\sOp$ can be viewed as the subcategory
of $\Op^{\Delta^{op}}$ such that the set of objects is constant in the simplicial direction,
one likewise has a fully faithful inclusion
$N\colon \sOp \to \mathsf{sdSet} = \mathsf{dSet}^{\Delta^{op}}$
with essential image those
$X \in \mathsf{sdSet}$
which both satisfy the strict Segal condition on each simplicial level and have constant object set, 
cf. Remark \ref{SIMPOPREM}.
Using the trick above,
one has the following compact description
of the simplicial operads $W(T) \in \sOp$ in \eqref{WXDEF_EQ}.

\begin{customprop}{A}\label{PROPA PROP}
The simplicial operad $W(T) \in \sOp$
(cf. \cite[(4.1)]{CM13b})
has nerve given by
\begin{equation}\label{PROPA EQ}
	\left(NW(T)\right)_{n}(S)
=
	\left\{
	\text{composable strings }
	S \xrightarrow{t} 
	J_0 \xrightarrow{i,p} 
	J_1 \xrightarrow{i,p} 
	\cdots \xrightarrow{i,p}
	J_n \xrightarrow{f,p}
	T
	\text{ of arrows in $\Omega$}
\right\}
\end{equation}
where we label maps in $\Omega$ as
$t/i/f/p$
to indicate they are 
tall/inner faces/faces/planar\footnote{
We expect most readers will be familiar with inner faces, faces, and planar maps. 
As for tall maps, they are defined 
as those maps in $\Omega$
that send the root to the root and leaves to leaves.
}
(cf. \S \ref{TREES_SEC}).
\end{customprop}
The description in \eqref{PROPA EQ}
makes heavy use of the standard factorization of maps in $\Omega$,
recalled in Proposition \ref{TREEFACT_PROP}.
As usual, the simplicial operators simply forget or replace the $J_i$.
Functoriality of \eqref{PROPA EQ} on both $S$ and $T$
is a consequence of the properties of said factorization,
and is described in \eqref{NWFUNC EQ},\eqref{NWTTSTAR EQ}.
Likewise, the properties that 
$NW(T)$ is levelwise Segal and 
has constant object set, also following from properties of the factorization, are discussed in Remark \ref{NWTNS_REM}.

For how \eqref{PROPA EQ}
recovers the original description of $W(T)$ in \cite[(4.1)]{CM13b},
see Example \ref{WREP_EX}.

\vskip 10pt

	By combining Proposition \ref{PROPA PROP} and \eqref{WXDEF_EQ}
	we now have a full definition of the functor
	$W_! \colon \mathsf{dSet} \to \sOp$,
	the explicit description of which is the goal of our main result, Theorem \ref{THMB}.

	Before stating that result, we need additional notation.
	For a necklace $\mathfrak{n}\colon J \to T$ as in Figure \ref{FIGURE},
	we write $\Omega[\mathfrak{n}] \in \dSet$
	for the dendroidal set in \eqref{DENDNECBAD EQ}
	(cf. Definitions \ref{NECKLACE_DEF} 
	and \ref{NECKREP_DEF}),
	and let $\mathsf{Nec} \subset \dSet$
	be the full subcategory spanned by the $\Omega[\mathfrak{n}]$. 
	The description of $W_!$ in Theorem \ref{THMB}
	will rely on a description 
	of $W_!(\Omega[\mathfrak{n}])$
	for $\mathfrak{n}$ a necklace 
	(this is elaborated on after \eqref{KEYSTRAT EQ}).
	The real appeal of \eqref{PROPA EQ}
	is then that it can easily be modified to describe
	$W_!(\Omega[\mathfrak{n}])$.
        
	Specifically,
	$N W_!(\Omega[\mathfrak{n}])$ is the subpresheaf of 
	$NW(T)$ in \eqref{PROPA EQ}
	obtained by imposing an additional condition
	which is closely related to the characterization of 
	maps between necklaces
	given in Proposition \ref{MAPNECK PROP}(ii).
	Should the map $S\to T$ in \eqref{PROPA EQ} be a tall map,
	this additional is that $J_0 \supseteq J$.
	Otherwise, one needs a more complex condition 
	$J_0 \supseteq J_{\overline{\phi S}}$
	(cf. Remark \ref{J0JFREPACK REM}),	
	related to 
	``outer faces'' of the necklace $\mathfrak{n}\colon J\to T$.
	See \eqref{NECKOUTF EQ} for a depiction of this notion
	of outer face.

The following is our main result, 
giving an explicit description of the
functor $W_!\colon \mathsf{dSet} \to \mathsf{sOp}$
based on \eqref{PROPA EQ}.

\begin{customthm}{B}[{cf. \cite[Thm. 1.3 and Cor. 4.4]{DS11}}]
	\label{THMB}
	Let $X\in \mathsf{dSet}$.
	Then $W_!(X) \in \sOp$ is the simplicial operad whose 
	nerve is described as follows.
	The simplices in the $n$-th level
	$NW_!(X)_{n}(S)$
	are equivalence classes of quadruples
	$(\mathfrak{n}, S \xrightarrow{\phi} T, \Omega[\mathfrak{n}] \xrightarrow{x} X, J_{\bullet})$ 
	where:
\begin{enumerate}[label=(\roman*)]
	\item $(\mathfrak{n}\colon J \to T) \in \mathsf{Nec}$ is a necklace; 
	\item $S \xrightarrow{\phi} T$
		is a tall map in $\Omega$
		such that $J \supseteq \phi(S)$;		
	\item $\Omega[ \mathfrak{n}] \xrightarrow{x} X$ is a map in $\mathsf{dSet}$;
	\item $J_{\bullet}$ denotes a 
		factorization of $\phi$ as below,
		and for which $J_0 \supseteq J$.
		Arrow labels have the same meaning as in 
		Proposition \ref{PROPA PROP}
		(note that the label of the last map differs from
		\eqref{PROPA EQ}).
\begin{equation}
\begin{tikzcd}
	S \ar{r}{t}
&
	J_0 \ar{r}{i,p}
&
	J_1 \ar{r}{i,p}
	&
\cdots
	\ar{r}{i,p}
&
	J_n \ar{r}{i,p}
&
	T
\end{tikzcd}
\end{equation}
\end{enumerate}
	The equivalence relation is generated by considering 
	$(\mathfrak{n},\phi,x,J_{\bullet})$ and
	$(\mathfrak{n}',\phi',x',J_{\bullet}')$
	to be equivalent if there is
	a map
	$\varphi \colon \Omega[\mathfrak{n}] \to \Omega[\mathfrak{n}']$
	(encoded by a map $\varphi \colon T \to T'$,
	 cf. Proposition \ref{MAPNECK PROP}(i))
	such that
	$\phi' = \varphi \phi$,
	$x = x' \varphi $
	and
	$J'_k = \varphi J_k$
	(i.e $J'_{\bullet}$
	is obtained by pushing 
	$J_{\bullet}$ along $\varphi$
	in the sense of \eqref{NWTTSTAR EQ}).
	
	Moreover, all such data have a representative, 
	\emph{unique up to isomorphism},
	for which:
	$J_{\bullet}$ is \emph{flanked},
	i.e. $J_0=J$ and $J_n=T$,
	and;
	$x$ is \emph{totally non-degenerate},
	i.e. for all beads 
	$T_b, b \in \boldsymbol{V}(J)$ of 
	$\mathfrak{n}$
	the dendrex
	$\Omega[T_b] \to \Omega[\mathfrak{n}] \to X$
	is non-degenerate. 
\end{customthm}

In the following, $\eta \in \Omega$
denotes the \emph{stick tree} with one edge and no vertices.

\begin{remark}\label{READMAPSP REM}
	The set of objects of $W_! X$ is simply the set
	$NW_! X(\eta) = X(\eta)$.
	Moreover, for each $X(\eta)$-signature, 
	i.e. tuple $(x_1,\cdots,x_n;x_0)$ with $x_i \in X(\eta)$,
	the space of maps
	$(W_! X)(x_1,\cdots,x_n;x_0) \in \mathsf{sSet}$
	is read off of Theorem \ref{THMB}
	by setting $S=C_n$ to be the $n$-corolla
	(i.e the tree with $n$ leaves and exactly one vertex)
	and restricting to those quadruples where the composite
	$\coprod_{\{0,1,\cdots,n\}}\Omega[\eta] 
	\to \mathsf{Sc}[C_n] \to 
	\Omega[\mathfrak{n}] \to X$
	is the signature $(x_1,\cdots,x_n;x_0)$.
\end{remark}

We now summarize the proof strategy for Theorem \ref{THMB},
which can be visualized by following diagram.
\begin{equation}\label{KEYSTRAT EQ}
\begin{tikzcd}
\Omega \ar[hookrightarrow]{r}
\ar{rrd}[swap]{W} 
&
\mathsf{Nec}
\ar[hookrightarrow]{r}
\ar{rd}{W}
&
\mathsf{dSet} \ar{rd}{W}
\ar[dashed]{d}{W}
\\
&&
\mathsf{sOp} \ar[hookrightarrow]{r}[swap]{N} &
\mathsf{sdSet}
\end{tikzcd}
\end{equation}
First, 
we extend \eqref{PROPA EQ} to a functor
$W \colon \mathsf{Nec} \to \sOp$
via direct construction
in Definition \ref{NWTNS DEF},
and then show that this functor is the left Kan extension of 
its restriction to 
$\Omega \hookrightarrow \mathsf{Nec}$,
cf. Proposition \ref{NECKCOL PROP}.
The point of this is then as follows.
Defining $W \colon \mathsf{dSet} \to \mathsf{sdSet}$
by making the right rhombus above into a left Kan extension diagram,
one has that:
\begin{enumerate*}
\item
$W \colon \mathsf{dSet} \to \mathsf{sdSet}$
actually lands in the essential image
of $N \colon \sOp \to \mathsf{sdSet}$,
cf. Proposition \ref{NWKANEX_PROP},
implicitly defining 
$W \colon \mathsf{dSet} \to \sOp$
and ensuring that the middle triangle is also a left Kan extension;
\item
the functor 
$W \colon \mathsf{dSet} \to \mathsf{sdSet}$
is easy to compute, due to being a left Kan extension onto a presheaf category,
so that the description in Theorem \ref{THMB}
is then mostly a matter of unpacking notation,
as done in Corollary \ref{NWXREPS COR}.
\end{enumerate*}
Crucially, we note that (i) would fail
if left Kan extending directly from 
$\Omega$ to $\dSet$. 
Lastly, the choice of the preferred 
flanked and totally non-degenerate 
representatives is addressed in 
Corollary \ref{NWXREPS2_COR}.

\section{Preliminaries}

\subsection{The category of trees}
\label{TREES_SEC}

We begin by recalling the Moerdijk-Weiss category $\Omega$ of trees
\cite{MW07}.
First, each object of $\Omega$ can be encoded by 
a (rooted) tree diagram $T$ as below.
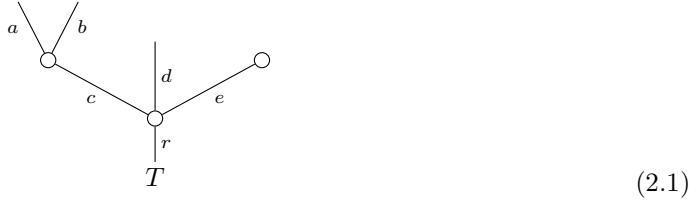
\begin{equation}\label{eq:TREE}
	\begin{tikzpicture}[auto,grow=up, level distance = 2.2em,
	every node/.style={font=\scriptsize,inner sep = 2pt}]%
	\tikzstyle{level 2}=[sibling distance=4em]%
	\tikzstyle{level 3}=[sibling distance=2.25em]%
            \node [font=\normalsize] {$T$}
            child{node [dummy] {}
              child{node [dummy] {}
                edge from parent node [swap] {$e$}
              }
              child[level distance = 2.9
              em]{edge from parent node [swap] {$d$}}
              child{node [dummy] {}
                child{edge from parent node [near end, swap] {$b$}}
                child{edge from parent node [near end] {$\phantom{b}a$}}
                edge from parent node {$c$}
              }
              edge from parent node [swap] {$r$}
            };        
      \end{tikzpicture}
\end{equation}
Edges with no vertices $\circ$ above them are called \textit{leaves}, the unique bottom edge is called the \textit{root},
and edges that are neither are called \textit{inner edges}.
In the example above, $a$, $b$ and $d$ are leaves, $r$ is the root, and $c$ and $e$ are inner edges.
The sets of edges, inner edges, and vertices of a tree $T$ are denoted 
$\boldsymbol{E}(T)$, 
$\boldsymbol{E}^{\mathsf{i}}(T)$, 
and $\boldsymbol{V}(T)$, respectively.

While the tree diagram description above is helpful for visualizing objects in $\Omega$,
in order to describe the arrows,
we will use the algebraic notion of
a \emph{broad poset},
originally due to Weiss \cite{Wei12}
and further developed in \cite{Per18},
which we now briefly recall.
For each edge $t$ in a tree topped by a vertex $\circ$, we write
$t^{\uparrow}$
for the tuple of edges immediately above $t$.
In \eqref{eq:TREE} one has  
$r^{\uparrow} = cde$, 
$c^\uparrow = ab$, 
and $e^\uparrow = \epsilon$,
where $\epsilon$ denotes the empty tuple.
Each vertex can then be encoded symbolically as
$t^{\uparrow} \leq t$,
which we call a 
\emph{generating broad relation}.
This notation is motivated by a form of transitivity.
For example,
in \eqref{eq:TREE}
the relations
$cde \leq r$ and $ab \leq c$
generate, under \emph{broad transitivity},
the relation $abde \leq r$,
and one may similarly obtain relations
$cd \leq r$ and $abd \leq r$.
These relations, together with identity relations $t \leq t$,
then form the \emph{broad poset associated with $T$}.

A map of trees $\varphi \colon S \to T$
in $\Omega$ is then an underlying map
of edge sets 
$\varphi \colon \boldsymbol{E}(S) \to \boldsymbol{E}(T)$
which preserves broad relations.

If an edge $t$ is pictorially above (or equal to) an edge $s$, we write $t \leq_d s$.
Equivalently, $t \leq_d s$ if there exists a broad relation $s_1\dots s_n \leq s$ such that $t = s_i$ for some $i$.\\

Our discussion will be simplified by assuming 
that $\Omega$ has exactly one representative of 
each \emph{planarized tree},
by which we mean a tree together with a planar representation as in \eqref{eq:TREE}.
Any map  
$\varphi \colon S \to T$ in $\Omega$
then has a unique factorization
$S \xrightarrow{\simeq} S' \to T$
as an isomorphism followed by a \emph{planar map}
\cite[Prop. 3.24]{BP_geo}.
In particular, the wide subcategory of $\Omega$ spanned by planar maps is skeletal,
i.e. the only planar isomorphisms are the identities.

\begin{notation}
	We write $\eta$ for the \textit{stick} tree, the unique tree with a single edge and no vertices.
\end{notation}

\begin{example}\label{TREEMAP_EX}
	The edge labels in each tree $S_i$ below determine maps
	$\boldsymbol{E}(S_i) \to \boldsymbol{E}(T)$,
	where $T$ is as in \eqref{eq:TREE}.
	For $i \leq 5$ this encodes maps
	$S_i \to T$ in $\Omega$,
	but not for $i=6$.
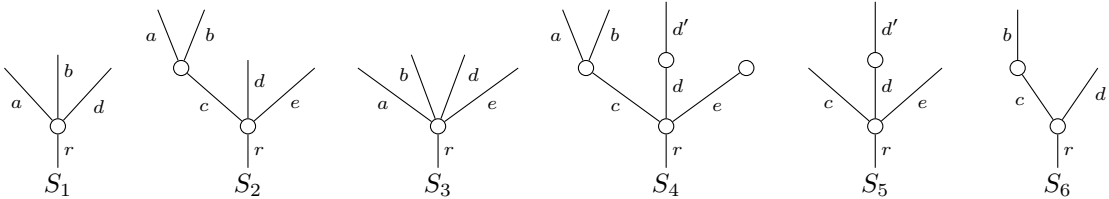
\begin{equation}
\begin{tikzpicture}[auto, grow=up, level distance = 2.2em,
	every node/.style={font=\scriptsize,inner sep = 2pt}]
\tikzstyle{level 2}=[sibling distance=2em]%
\tikzstyle{level 3}=[sibling distance=2.25em]%
\node at (0,0) [font=\normalsize] {$S_1$} %
	child{node [dummy] {}
		child{edge from parent node [swap] {$d$}}
		child [level distance = 2.7em] {edge from parent node [swap,near end] {$b$}}
		child{edge from parent node {$a$}}
	edge from parent node [swap] {$r$}
	};
\tikzstyle{level 2}=[sibling distance=2.5em]%
\tikzstyle{level 3}=[sibling distance=1.75em]%
\node at (2.5,0) [font=\normalsize] {$S_2$} 
	child{node [dummy] {}
		child{
		edge from parent node [swap] {$e$}
		}
		child[level distance = 2.5em]{
		edge from parent node [swap,pos=0.65] {$d$}
		}
		child{node [dummy] {}
			child{edge from parent node [swap,near end] {$b$}}
			child{edge from parent node [near end] {$\phantom{b}a$}}
		edge from parent node {$c$}
		}
	edge from parent node [swap] {$r$}
	};
\tikzstyle{level 2}=[sibling distance=2em]%
\node at (5,0) [font=\normalsize] {$S_3$} 
	child{node [dummy] {}
		child{
		edge from parent node [swap] {$e\phantom{a}$}
		}
		child [level distance = 2.7em]{edge from parent node [swap, very near end] {$d$}}
		child [level distance = 2.7em]{edge from parent node [very near end] {$b$}}
		child{edge from parent node {$\phantom{e}a$}}
	edge from parent node [swap] {$r$}
	};
\tikzstyle{level 2}=[sibling distance=3em]%
\node at (8,0) [font=\normalsize] {$S_4$} 
	child{node [dummy] {}
		child{node [dummy] {}
		edge from parent node [swap] {$e$}
		}
		child[level distance = 2.5em]{node [dummy] {}
			child[level distance = 2.2em]{
			edge from parent node [swap] {$d'$}}
		edge from parent node [swap,pos=0.65] {$d$}
		}
		child{node [dummy] {}
			child{edge from parent node [swap,near end] {$b$}}
			child{edge from parent node [near end] {$\phantom{b}a$}}
		edge from parent node {$c$}
		}
	edge from parent node [swap] {$r$}
	};
\tikzstyle{level 2}=[sibling distance=2.5em]%
\node at (10.75,0) [font=\normalsize] {$S_5$} %
	child{node [dummy] {}
		child{edge from parent node [swap] {$e$}}
		child[level distance = 2.5em]{node [dummy] {}
			child[level distance = 2.2em]{edge from parent node [swap] {$d'$}}
		edge from parent node [swap,pos=0.65] {$d$}
		}
		child{edge from parent node {$c$}}
	edge from parent node [swap] {$r$}
	};
\tikzstyle{level 2}=[sibling distance=3em]%
\node at (13.15,0) [font=\normalsize] {$S_6$}
	child{node [dummy] {}
		child{edge from parent node [swap,near end] {$d\phantom{c}$}}
		child{node [dummy] {}
			child{edge from parent node {$b$}}
		edge from parent node [near end] {$\phantom{d}c$}
		}
	edge from parent node [swap] {$r$}
	};
\end{tikzpicture}
\end{equation}
\end{example}

\begin{definition}
        \label{TREEMAP_DEF}
        A map of trees $\varphi \colon S \to T$ is called:
        \begin{itemize}
        \item a \textit{tall map} if
                $\varphi(\underline{l}_S) = \underline{l}_T$ and $\varphi(r_S) = r_T$,
                with $\underline{l}_{(-)}$ and $r_{(-)}$ denoting the tuple of leaf edges and the root edge;
        \item a \textit{face map} if it is injective on edges;
                an \textit{inner face} if it is also tall; and
                an \textit{outer face} if, for any factorization
                $\varphi \simeq \varphi_1\varphi_2$
                with $\varphi_1,\varphi_2$ face maps
                and $\varphi_2$ inner, 
                $\varphi_2$ is an isomorphism;
        \item a \textit{degeneracy} if it is surjective on edges and preserves leaves
                (and is thus tall);
        \item a \textit{convex map} if,
        	whenever $e <_d e' <_d e''$ in $T$ 
        	and $e,e''$ are in the image of $\varphi$,
        	then so is $e'$.
        \end{itemize}
\end{definition}

Pictorially, inner face maps 
$S \to T$ remove some edges in $T$
(and merge the vertices adjacent to those edges),
outer face maps remove some vertices of $T$,
and degeneracies collapse some of the unary vertices of $S$.
Face maps combine inner and outer faces,
tall maps combine inner faces and degeneracies, 
and convex maps combine outer faces and degeneracies (cf. Remark \ref{TREEFACTNAMES_REM}).

\begin{example}
	In Example \ref{TREEMAP_EX},
	$S_1 \to T$ is an inner face,
	$S_2 \to T$ is an outer face,
	$S_3 \to T$ is a face that is neither inner nor outer,
	$S_4 \to T$ is a degeneracy,
	and $S_5 \to T$ is a convex map.
\end{example}

\begin{notation}\label{MAPLABELS_NOT}
	Throughout the remainder of the paper,
        we will label a 
        map in $\Omega$
        by the letters d/i/o/t/f/p
        to indicate that the map is
        a degeneracy/inner face/outer face/tall/face/planar.
\end{notation}

\begin{proposition}[{\cite[Prop. 2.2]{BP_edss}}]
      \label{TREEFACT_PROP}
      A map of trees $\varphi \colon S \to T$ 
      has a strictly unique factorization
      \begin{equation}\label{TREEFACT_EQ}
              S \xrightarrow{\simeq}
              S_p \xrightarrow{pd} 
              \varphi S \xrightarrow{pi} 
              \overline{\varphi S} \xrightarrow{po} T
      \end{equation}
      as an isomorphism followed by a planar degeneracy, a planar inner face, and a planar outer face.
\end{proposition}

\begin{notation}\label{TREEFACT NOT}
      The notation $\varphi S$ is motivated by the fact that this tree has edge set
      $\boldsymbol{E}(\varphi S) = \varphi (\boldsymbol{E}(S))$,
      while the 
      notation $\overline{\varphi S}$ is an instance of the 
      \emph{outer closure of a face}
      notation in \cite[Not. 2.14]{BP_edss}
      which, for a face $F$, defines $\overline{F}$
      as the smallest (planar) outer face containing $F$.
\end{notation}

\begin{remark}[{cf. \cite[Rems. \ref{TAS-TODF REM}, \ref{TAS-CNVXM REM}, \ref{TAS-TREEMAPCOMP_REM}]{BP_TAS}}]
        \label{TREEFACTNAMES_REM}
    For any subset
	$\mathcal{S} \subseteq \{\simeq,pd,pi,po\}$
	of the arrow labels 
	in \eqref{TREEFACT_EQ},
	the type of maps whose
	factors labeled by $\mathcal{S}$ are identities 
	is closed under composition.

        In particular, as (non-planar) tall maps (resp. face maps, convex maps) are characterized as those maps such that
        the component labeled $po$ (resp. $pd$, $pi$)
        in \eqref{TREEFACT_EQ}
        is the identity,
        we have that these types of maps (and their planar analogues) are closed under composition.
\end{remark}

\begin{remark}\label{IGNPL REM}
	Modifying \eqref{TREEFACT_EQ} by ignoring planarity
	gives a factorization 
	$S \xrightarrow{d} U \xrightarrow{i} V
	\xrightarrow{o} T$,
	unique up to unique isomorphisms.
	Moreover, combining the $i$ and $o$ arrows
	recovers the usual degeneracy-face decomposition
	in \cite[Lemma 3.1]{MW07},
	while combining the $d$ and $i$ arrows recovers the 
	tall-outer decomposition in \cite[Prop. 3.36]{BP_geo}.
\end{remark}

\begin{notation}\label{LEAFROOT NOT}
	A \textit{corolla} is a tree with a single vertex.
	For each $n\geq 0$, one has a corolla $C_n$ with $n$ leaves,
	and we write $\Sigma$ for the category of corollas and isomorphisms, which is naturally identified with the category of standard finite sets
	$\{1,2,\cdots,n\}$ and isomorphisms.
	
	For any tree $T$ with $n$ leaves, 
	we write $\mathsf{lr}(T)$,
	which we call the \emph{leaf-root of $T$},
	for the corolla $C_n$,
	which comes together with a unique planar tall map
	$\mathsf{lr}(T) \to T$.
\end{notation}

\begin{example}
	For the tree $T$ in \eqref{eq:TREE},
	the corolla $\mathsf{lr}(T)$ is $S_1$
	in Example \ref{TREEMAP_EX}.
\end{example}

\begin{notation}\label{TV_NOT}
	For a tree $T$ and $v \in \boldsymbol{V}(T)$,
	we write $T_v \to T$ for the planar outer face consisting of only this vertex and its adjacent edges.
	Further, for a map 
	$\varphi \colon J\to T$ and $b \in \boldsymbol{V}(J)$,
	we write
	$T_b = \overline{\varphi J_b}$.
	Compare with the notion of \emph{bead}
	in Definition \ref{NECKLACE_DEF}(ii)
	and Figure \ref{FIGURE}.
\end{notation}

\subsection{Dendroidal sets and operads}
\label{DSETSOPS SEC}

This subsection recalls the definitions
of the key categories appearing in the main adjunctions \eqref{SOPDSET_EQ}.
First, the category of 
\emph{dendroidal sets}
is the category $\dSet = \Set^{\Omega^{op}}$
of presheaves on $\Omega$.

There are a number of presheaves that play a key role in the theory of dendroidal sets.
First, for each tree $T\in \Omega$,
one has the representable presheaf
$\Omega[T](S) = \Omega(S,T)$.
Moreover, one has the following subpresheaves of 
$\Omega[T]$,
called the \emph{boundary},
\emph{inner horn},
and \emph{Segal core} 
\begin{equation}\label{SUBPRESH EQ}
\partial \Omega[T] = \bigcup_{U \in \mathsf{Face}(T), U\neq T} \Omega[U],
	\qquad
\Lambda^E[T] = \bigcup_{U \in \mathsf{Face}(T),
	U \not \hookrightarrow T-E} \Omega[U],
	\qquad
Sc[T] = \bigcup_{U \in \mathsf{Face}_{sc}(T)} \Omega[U],
\end{equation}
where $\mathsf{Face}(T)$ is the poset of planar faces,
$\emptyset \neq E \subseteq \boldsymbol{E}^{\mathsf{i}}(T)$
is a non-empty set of inner edges,
and 
$\mathsf{Face}_{sc}(T)$ is the poset of planar outer faces with no inner edges
(i.e. $U$ with either a single edge or a single vertex).
Typically $\partial \Omega[T]$ and $\Lambda^E[T]$
are the main objects of interest
(see, e.g. \cite[\S 4]{BP_edss}, for further discussion),
but in this paper the $Sc[T]$ play the central role,
partly due to $Sc[T]$ being a necklace, 
cf. Remark \ref{SEGISNEC REM},
and partly since they appear in the Segal condition below.

Given $X,A \in \mathsf{dSet}$,
let us abbreviate $X(A) = \mathsf{dSet}(A,X)$.
We then say that $X$ satisfies the 
\emph{strict Segal condition} if,
for any tree $T \in \Omega$,
the natural map below is an isomorphism.
\begin{equation}\label{SCV1 EQ}
X(T) = X(\Omega[T]) \xrightarrow{\simeq} X(Sc[T])
\end{equation}
As noted at the start of \S \ref{MAINRESULT_SEC},
we will identify the category $\Op$ of (colored) operads
with its essential image under the nerve 
$N \colon \Op \to \mathsf{dSet}$,
which consists of the objects satisfying 
the strict Segal condition \eqref{SCV1 EQ}
(more precisely, this follows from 
\cite[Prop. 5.3 and Thm. 6.1]{MW09}
together with Remark \ref{SSCOTHER REM} below).
For the usual description of $\Op$, see 
\cite[\S 1]{CM13b} or \cite[Def. \ref{OC-FREEOP DEF}]{BP_FCOP}.

Some of our arguments in \S \ref{WCONS SEC}
will be simplified by using an alternative formulation of \eqref{SCV1 EQ},
which is motivated by the fact that colored operads
$\Op$ are most commonly defined using colored trees.
As such, we first recall the following, cf. 
\cite[Def. \ref{OC-COLFOR DEF}]{BP_FCOP}.

\begin{definition}\label{CTREE_DEF}
	Let $\mathfrak C$ be a set of colors.
	The category $\Omega_{\mathfrak C}$ of \textit{$\mathfrak C$-trees} has objects pairs $(T, \mathfrak c)$ with 
	$T \in \Omega$ a tree and
	$\mathfrak c \colon \boldsymbol{E}(T) \to \mathfrak C$ a coloring of its edges,
	and arrows
	$(S, \mathfrak d) \to (T, \mathfrak c)$
	given by maps
	$\varphi \colon S \to T$ in $\Omega$ such that $\mathfrak d = \mathfrak c \varphi$.
	
	Moreover, any a map of color sets
	$f \colon \mathfrak C \to \mathfrak D$
	induces a functor
	$f\colon \Omega_{\mathfrak{C}} \to \Omega_{\mathfrak{D}}$
	via
	$(T,\mathfrak{c}) \mapsto (T,f\mathfrak{c})$.
\end{definition}

\begin{notation}\label{MAPSOVCOL NOT}
	Given $X \in \dSet$, tree $T \in \Omega$, 
	and coloring $\mathfrak c \colon \boldsymbol{E}(T) \to X(\eta)$,
	we write $X_{\mathfrak c}(T) \in \Set$ for the pullback below.
	\begin{equation}
	\label{XUDECALT EQ}
	\begin{tikzcd}
	X_{\mathfrak c}(T) \arrow[r] \arrow[d]
	&
	X(T) \arrow[d]
	\\
	\** \arrow{r}[swap]{\mathfrak c}
	&
	\displaystyle{
		\prod\limits_{\boldsymbol{E}(T)} X(\eta)
	}
	\end{tikzcd}
	\end{equation}
\end{notation}

\begin{remark}
	The notation above gives a decomposition
	$X(T) \simeq \coprod_{\{\mathfrak c \colon \boldsymbol{E}(T) \to X(\eta)\}} X_{\mathfrak c}(T)$
	for any $X\in \dSet$.
	Moreover, the assignment
	$(T,\mathfrak c) \mapsto X_{\mathfrak{c}}(T)$ 
	is functorial on
	$X(\eta)$-trees $(T,\mathfrak c) \in \Omega^{op}_{X(\eta)}$,
	so that $X$ has an equivalent description
	as a presheaf on $\Omega_{X(\eta)}$.
	In fact, a little more is true. 
	If one writes 
	$\mathsf{dSet}_{\mathfrak{C}} \subset \dSet$
	for the subcategory of those $X$ such that
	$X(\eta) = \mathfrak{C}$ and maps that are the identity on 
	$X(\eta)$, there is an equivalence of categories ({cf. \cite[ \eqref{TAS-PREOPCOLFIXEQ EQ}]{BP_TAS}})
\begin{equation}\label{DSETCID EQ}
\dSet_{\mathfrak C} \xrightarrow{\simeq}
 \Fun_{\**}(\Omega_{\mathfrak C}^{op}, \Set),
\qquad
\big(T \mapsto X(T)\big) \mapsto \big( (T, \mathfrak c) \mapsto X_{\mathfrak c}(T) \big)
\end{equation}
where $\Fun_{\**}$ denotes \emph{pointed functors},
i.e. functors $X$ such that
$X_{\mathfrak c}(\eta) = \**$
for any $\mathfrak{c} \in \mathfrak{C}$.
\end{remark}

Using the $X_{\mathfrak c}(T)$ notation, 
the Segal condition in \eqref{SCV1 EQ}
then decomposes into isomorphisms
	\begin{equation}\label{SCV2 EQ}
	X_{\mathfrak c}(T) \xrightarrow{\simeq} \prod_{v \in \boldsymbol{V}(T)} X_{\mathfrak c_v}(T_v)
	\end{equation}
for any $X(\eta)$-tree $(T,\mathfrak{c})$,
and with $\mathfrak{c}_v$
the restricted coloring
$\boldsymbol{E}(T_v) \to \boldsymbol{E}(T) \to X(\eta)$.

\begin{example}\label{OMTSEG EX}
	Representables $\Omega[S]$ satisfy
	the strict Segal condition \eqref{SCV2 EQ}.
	Explicitly, this Segal condition says
	that a map
	$T\to S$ in $\Omega$ is determined by maps
	$T_v\to S$, 
	which is the content of \cite[Prop. 5.11]{Per18}.
	The operad $\Omega(S)$
	such that $\Omega[S] = N \Omega(S)$
	is defined in \cite[\S 3]{MW07}.
\end{example}

\begin{remark}\label{PULLBACKS REM}
Given a map of colors 
$f \colon \mathfrak{C} \to \mathfrak{D}$,
the identification \eqref{DSETCID EQ}
and precomposition with 
$f\colon \Omega_\mathfrak{C} \to \Omega_\mathfrak{D}$
yield the left functor $f^{\**}$ below.
Moreover, $f^{\**}$ is clearly compatible with the Segal condition
\eqref{SCV2 EQ} so that, writing
$\Op_{\mathfrak{C}} = \dSet_{\mathfrak{C}} \cap \Op$,
one has the restricted $f^{\**}$ functor on the right.
\[
f^{\**} \colon \dSet_{\mathfrak{D}} \to \dSet_{\mathfrak{C}}
\qquad \qquad
f^{\**} \colon \Op_{\mathfrak{D}} \to \Op_{\mathfrak{C}}
\]
Note that maps $X \to Y$ in either $\dSet$ or $\Op$
over a color map $f$ are then in bijection with fixed color maps
$X \to f^{\**} Y$.
\end{remark}

\begin{remark}\label{SSCOTHER REM}
	Condition \eqref{SCV1 EQ} has other formulations.
	Indeed, by \cite[Props. 3.22, 3.31]{BP_edss}
	one may replace the Segal cores $Sc[T]$
	in \eqref{SCV1 EQ} with the inner horns $\Lambda^E[T]$,
	thus saying that $X$ has the strict
	right lifting property against the maps
	$\Lambda^E[T] \to \Omega[T]$.
	This shows that $\infty$-operads generalize operads,
	as the first are defined by the 
	non-strict version of this property \cite[\S 5]{MW09}.
\end{remark}

\begin{remark}\label{SIMPOPREM}
	When dealing with simplicial operads $\sOp$,
	we will also have need to discuss
	simplicial dendroidal sets
	$\mathsf{sdSet} = \mathsf{Set}^{\Delta^{op} \times \Omega^{op}}$,
	whose levels we write as
	$X_n(T)$ for $T \in \Omega$ and $[n] \in \Delta$.
	As noted in \S \ref{MAINRESULT_SEC},
	applying the nerve along each simplicial direction
	yields a fully faithful inclusion
	$N \colon \sOp \to \mathsf{sdSet}$
	with essential image those $X \in \mathsf{sdSet}$
	for which $X(\eta)$ is a discrete simplicial set
	and which satisfy the Segal condition 
	\eqref{SCV1 EQ},\eqref{SCV2 EQ}
	on each simplicial level
	(or equivalently, which satisfy
	\eqref{SCV1 EQ},\eqref{SCV2 EQ}
	when regarded as an identification of simplicial sets).
\end{remark}

\section{Dendroidal necklaces}

We now formalize the notion of dendroidal necklace
discussed in the introduction,
cf. Figure \ref{FIGURE}, 
thus generalizing the key notion in \cite{DS11}.
For the meaning of 
$\overline{\mathfrak{n}J_b}$, $T_b$, see
Notations \ref{TREEFACT NOT}, \ref{TV_NOT}.

\begin{definition}[{cf. \cite[\S 3]{DS11}}]
        \label{NECKLACE_DEF}
	A \emph{necklace} is 
	a planar inner face map
	$\mathfrak{n} \colon J \to T$
	in $\Omega$.
	Moreover:
	\begin{enumerate}[label = (\roman*)]
		\item 
		$J$ is called the \emph{inner face of joints} of the necklace;
		\item for each vertex $b \in \boldsymbol{V}(J)$,
		the outer face
		$\overline{\mathfrak{n} J_b} = T_b \hookrightarrow T$
		is called a \emph{bead} of the necklace,
		and we write
		$\boldsymbol{B}(\mathfrak{n}) 
		\simeq 
		\boldsymbol{V}(J)$
		for the set of beads.
	\end{enumerate}
\end{definition}

\begin{example}
	For $\mathfrak{n} \colon J \to T$ the necklace 
	in Figure \ref{FIGURE} the beads 
	are the trees $T_i$ depicted therein.
\end{example}

We now formalize the idea behind \eqref{DENDNECBAD EQ},
thus defining the presheaves $\Omega[\mathfrak{n}]$.
Recall, cf. \eqref{SUBPRESH EQ}, 
that the Segal core poset
$\mathsf{Face}_{sc}(J)$
consists of the planar outer faces of $J$ with no inner edges.

\begin{definition}
        \label{NECKREP_DEF}
	Given a necklace $\mathfrak{n} \colon J \to T$ 
	we define its representable presheaf
	$\Omega[\mathfrak{n}] \in \mathsf{dSet}$ by
	\begin{equation}
	\Omega[\mathfrak{n}] 
	= 
	\underset{U \in \mathsf{Face}_{sc}(J)}{\colim}
	\Omega[\overline{\mathfrak{n} U}]
	=
	\bigcup_{U \in \mathsf{Face}_{sc}(J)} 
	\Omega[\overline{\mathfrak{n} U}]
	\end{equation}
	where the union formula is taken inside $\Omega[T]$.
	
	The category $\mathsf{Nec}$ of necklaces is then the full subcategory of $\mathsf{dSet}$
	spanned by the $\Omega[\mathfrak{n}]$.
\end{definition}

\begin{remark}\label{GRAFT REM}
For any tall map $S \to T$ in $\Omega$,
\cite[Cor. 3.75]{BP_geo} says that
one has a decomposition
\[T= 
\underset{U \in \mathsf{Face}_{sc}(S)}{\colim}
\overline{\mathfrak{n} U}
\]
as a colimit in $\Omega$, 
which formalizes the grafting procedure in Figure \ref{FIGURE}.
Crucially, the relevance of Definition \ref{NECKREP_DEF} 
comes from the fact that the Yoneda $\Omega[-]$
does not preserve this decomposition.
\end{remark}

\begin{remark}\label{SEGISNEC REM}
	The $\Omega[\mathfrak{n}]$ presheaves for necklaces
	$\mathfrak{n}\colon J \to T$
	interpolate between the representable and 
	Segal core presheaves $\Omega[T]$ and $Sc[T]$
	in \S \ref{DSETSOPS SEC}.
	More explicitly,
	each tree $T \in \Omega$
	gives rise to necklaces
	$T \xrightarrow{=} T$ and
	$\mathsf{lr}(T) \to T$
	(cf. Notation \ref{LEAFROOT NOT})
	for which
\[
	\Omega[T \xrightarrow{=} T] = Sc[T],
\qquad
	\Omega[\mathsf{lr}(T) \to T] = \Omega[T].
\]
        In particular,
        one obtains a natural inclusion
        $\Omega \hookrightarrow \mathsf{Nec}$
        given by $T \mapsto \left( \mathsf{lr}(T) \to T \right)$.
        However, we caution that the assignment
        $T \mapsto Sc[T]$ is not functorial on $T$
        (more precisely, it is functorial only with respect to \emph{convex} maps of trees,
        in the sense of Definition \ref{TREEMAP_DEF}). 
\end{remark}

\begin{remark}\label{SEGFORNECK REM}
	If $X\in \dSet$ satisfies the Segal condition \eqref{SCV2 EQ}
	and $\mathfrak{n}\colon J \to T$ is a necklace, then	
	\[
                X_{\mathfrak{c}}(T)
                \simeq 
                \prod_{v \in \boldsymbol{V}(T)}
                X_{\mathfrak{c}_v}(T_v)
                \simeq
                \prod_{b \in \boldsymbol{B}(\mathfrak{n})}
                \prod_{v \in \boldsymbol{V}(T_b)}
                X_{\mathfrak{c}_v}(T_v)
                \simeq
                \prod_{b \in \boldsymbol{B}(\mathfrak{n})}
                X_{\mathfrak{c}_b}(T_b),
	\]
	so that the natural maps
	$X(T) \xrightarrow{\simeq} X(\Omega[\mathfrak{n}])$
	are isomorphisms, cf. \eqref{SCV1 EQ}.
\end{remark}


\begin{lemma}\label{FACEINNECK LEM}
	Let $\mathfrak{n} \colon J \to T$ be a necklace. Then
	\begin{enumerate}[label=(\roman*)]
		\item a face $U \hookrightarrow T$
		is in $\Omega[\mathfrak{n}]$
		iff its outer closure $\overline{U}$ is; 
		\item an outer face 
		$U = \overline{U} \hookrightarrow T$
		is in $\Omega[\mathfrak{n}]$ iff 
		$\boldsymbol{E}^{\mathsf{i}}(J) \cap 
		\boldsymbol{E}^{\mathsf{i}}(U) = \emptyset$;
		\item there is a decomposition
		$
		\boldsymbol{E}(T) \simeq
		\boldsymbol{E}(J) \amalg 
		\coprod_{b \in \boldsymbol{V}(J)}
		\boldsymbol{E}^{\mathsf{i}}(T_b)
		= 
		\boldsymbol{E}(J) \amalg 
		\coprod_{b \in \boldsymbol{B}(\mathfrak{n})}
		\boldsymbol{E}^{\mathsf{i}}(T_b)
		$.
	\end{enumerate}
\end{lemma}

\begin{proof}
	(i) follows since $\Omega[\mathfrak{n}]$ is an union of outer faces.
	
	The arguments for (ii),(iii) are by induction on the number of inner edges $\boldsymbol{E}^{\mathsf{i}}(J)$,
	with the base case $\boldsymbol{E}^{\mathsf{i}}(J) = \emptyset$,
	so that $J=T=\eta$,
	being obvious.
	Otherwise, letting $e \in \boldsymbol{E}^{\mathsf{i}}(J)$, since $e$ is an inner edge of
	both $J$ and $T$
	one has grafting decompositions
	$J = J' \amalg_e J''$,
	$T = T' \amalg_e R''$
	together with inner face maps
	$\mathfrak{n}' \colon J' \to T'$,
	$\mathfrak{n}'' \colon J'' \to T''$.
	One then has that 
	$U$ is in $\Omega[\mathfrak{n}]$
	iff it is in either
	$\Omega[\mathfrak{n}']$ or in $\Omega[\mathfrak{n}'']$,
	yielding the induction step for (ii).
	The induction step for (iii) likewise follows. 
\end{proof}

\begin{remark}
	If $S \xrightarrow{d} S'$
	is a degeneracy,
	the vertices of $S'$ are identified with
	the vertices of $S$ that are not collapsed to edges.
	Thus, by factoring a tall map
	$\varphi \colon S \xrightarrow{t} T$ as
	a degeneracy followed by inner face
	$S \xrightarrow{d} \varphi S \xrightarrow{i} T$,
	cf. Remark \ref{TREEFACTNAMES_REM},
	the decomposition (iii) in Lemma \ref{FACEINNECK LEM}
	generalizes to 
\begin{equation}\label{EDGEBREAK EQ}
	\boldsymbol{E}(T) = 
	\boldsymbol{E}(\varphi S) \amalg 
	\coprod_{v \in \boldsymbol{V}(S)}
	\boldsymbol{E}^{\mathsf{i}}(\overline{\varphi S_v}).
\end{equation}
\end{remark}

\begin{notation}\label{JFNOT NOT}
	Given a necklace
	$\mathfrak{n} \colon J \to T$
	and outer face $F \to T$
	we write 
	$\mathfrak{n}_F \colon J_F \to F$
	for the necklace characterized by
\[
	\boldsymbol{E}^{\mathsf{i}}(J_F)
	=
	\boldsymbol{E}^{\mathsf{i}}(J)
	\cap
	\boldsymbol{E}^{\mathsf{i}}(F).
\]
\end{notation}

\begin{example}
	Letting $\mathfrak n \colon J \to T$ be the necklace in 
	Figure \ref{FIGURE},
	and for $F \to T$ the outer face depicted 
	in the middle below,
	the following represents 
	$\mathfrak{n}_F \colon J_F \to F$,
	with the $F_i$ its beads.
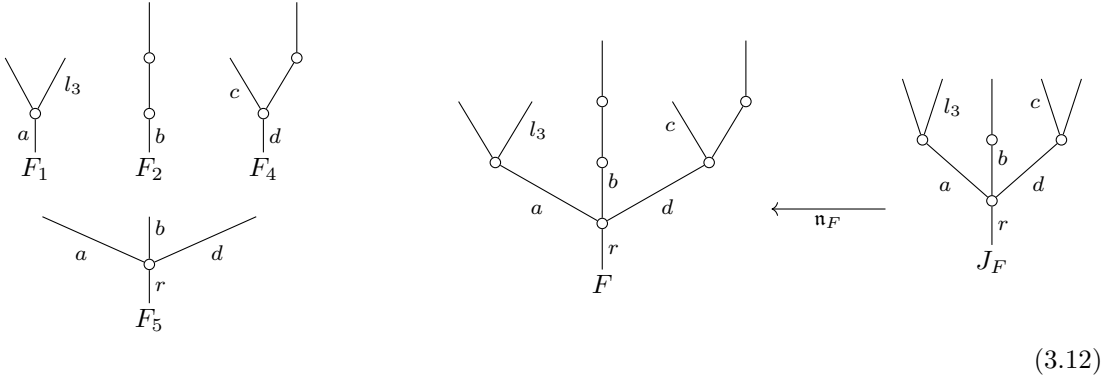
\begin{equation}\label{NECKOUTF EQ}
	\begin{tikzpicture}[grow=up,auto,level distance=2.1em,
	every node/.style = {font=\footnotesize,inner sep=2pt},
	dummy/.style={circle,draw,inner sep=0pt,minimum size=1.375mm}]
	\begin{scope}[xshift=-2em]
	\begin{scope}
	\tikzstyle{level 2}=[sibling distance=2.25em]%
	\tikzstyle{level 3}=[sibling distance=1.25em]%
	\node at (-0.5,1) [font = \normalsize] {$F_1$}
	child{node [dummy] {}
		child{
		edge from parent node [near end,swap]{$l_3$}}
		child
		edge from parent node {$a$}};
	\end{scope}
	\node at (1,1) [font = \normalsize] {$F_2$}
	child{node [dummy] {}
		child{node [dummy] {}
			child
		}
		edge from parent node [swap] {$b$}};
	\begin{scope}
	\tikzstyle{level 2}=[sibling distance=2.5em]%
	\node at (2.5,1) [font = \normalsize] {$F_4$}
	child{node [dummy] {}
		child{node[dummy] {}
			child
		}
		child{
			edge from parent node {$c$}}
		edge from parent node [swap] {$d$}};
	\end{scope}
	\begin{scope}
	\tikzstyle{level 2}=[sibling distance=4.25em,level distance=1.8em]%
	\node at (1,-1) [font = \normalsize] {$F_5$}
	child{node [dummy] {}
		child[sibling distance = 4em]{edge from parent node [swap] {$d$} }
		child[sibling distance = 3.5em]{edge from parent node [swap,near end] {$b$} }
		child[sibling distance = 4em]{ edge from parent node {$\phantom{d}a$} }
		edge from parent node [swap] {$r$}};
	\end{scope}
	\end{scope}
	\begin{scope}[yshift=-1.5em,xshift=2.5em]
	\begin{scope}[level distance=2.3em]
	\tikzstyle{level 2}=[sibling distance=4em]%
	\tikzstyle{level 3}=[sibling distance=2.75em]%
	\tikzstyle{level 4}=[sibling distance=1.25em]%
	\tikzstyle{level 5}=[sibling distance=0.875em]%
	\node at (5.375,0) [font = \normalsize] {$F$}
	child{node [dummy] {}
		child[sibling distance = 4em]{node [dummy] {}
			child{node [dummy] {}
				child
			}
			child{
				edge from parent node [near end] {$c$}}
			edge from parent node [swap] {$d$}}
		child[sibling distance =3.5em]{node [dummy] {}
			child{node [dummy] {}
				child
			}
			edge from parent node [swap,near end] {$b$}}
		child[sibling distance =4em]{node [dummy] {}
			child{
			edge from parent node [near end,swap]{$l_3$}}
			child
			edge from parent node {$\phantom{d}a$}}
		edge from parent node [swap] {$r$}};
	\end{scope}
	\begin{scope}[level distance=2.3em]
	\tikzstyle{level 2}=[sibling distance=2.6em]%
	\tikzstyle{level 3}=[sibling distance=1.5em]%
	\node at (10.5,0.3) [font = \normalsize] {$J_F$}
	child{node [dummy] {}
		child{node [dummy] {}
			child
			child{
				edge from parent node [near end] {$c$}}	
			edge from parent node [swap] {$d$}}
		child{node [dummy] {}
			child
			edge from parent node [near end,swap] {$b$}}
		child{node [dummy] {}
			child{
			edge from parent node [near end,swap]{$l_3$}}
			child
			edge from parent node {$\phantom{d}a$}}
		edge from parent node [swap] {$r$}};
	\end{scope}
	\draw [->] (9.1,1) -- node {$\mathfrak{n}_F$} (7.6,1);
	\end{scope}
	\end{tikzpicture}
\end{equation}
In general, the $\mathfrak{n}_F$ construction works as follows, where say a bead $T_b$ is \emph{outer}
if it is connected to at most one other bead
$T_{b'}$
(equivalently, if all outer edges of the bead
$T_b$ are outer edges of $T$ itself, \emph{except at most one}).
First, $\mathfrak{n}_F$ removes some outer beads altogether.
In this example, 
$T_3$ from Figure \eqref{FIGURE} is removed.
Then, some of the resulting outer beads are replaced with outer faces of themselves.
In this example, 
$T_1$, $T_2$, $T_4$ from Figure \eqref{FIGURE}
are replaced with $F_1$, $F_2$, $F_4$
(note that $T_4$ was initially not an outer bead, 
but became so upon removal of $T_3$).
	
We caution that, just as in this example,
one in general does not have a map $J_F \to J$,
as $\boldsymbol{E}(J)$ needs not contain $\boldsymbol{E}(J_F)$.
Instead, as will follow from
Proposition \ref{MAPNECK PROP}(ii),
one has a map of necklaces
$\mathfrak{n}_F \to \mathfrak{n}$,
which should be thought of as an outer face map in $\mathsf{Nec}$.
\end{example}

\begin{corollary}\label{NECINT COR}
	Let $\mathfrak{n} \colon J \to T$ be a necklace and
	$F \to T$ be an outer face.
	Then
\[
	\Omega[\mathfrak{n}_F] = \Omega[\mathfrak{n}] \cap \Omega[F]
\]
	where the intersection is taken 
	as subpresheaves of $\Omega[T]$.
\end{corollary}

\begin{proof}
	Combining (i),(ii) in 
	Lemma \ref{FACEINNECK LEM}
	we see that a face $U \hookrightarrow F$ is
	in $\Omega[\mathfrak{n}]$
	iff $\boldsymbol{E}(J) \cap 
	\boldsymbol{E}^{\mathsf{i}}(\overline{U}) = \emptyset$,
	where (since $F$ is outer) the outer closure $\overline{U}$
	can be taken in either $T$ or $F$.
	But,
	since $\overline{U} \hookrightarrow F$
	implies 
	$\boldsymbol{E}^{\mathsf{i}}(\overline{U})
	\subseteq 
	\boldsymbol{E}^{\mathsf{i}}(F)$,
	this is equivalent to 
	$\boldsymbol{E}(J_F) \cap 
	\boldsymbol{E}^{\mathsf{i}}(\overline{U}) = \emptyset$,
	i.e. to $U$ being in $\Omega[\mathfrak{n}_F]$.
\end{proof}

We next characterize the maps in $\mathsf{Nec}$.
See Notations \ref{TREEFACT NOT}, \ref{JFNOT NOT}
for the meaning of
$\varphi J$,
$J'_{\overline{\varphi J}}$.

\begin{proposition}\label{MAPNECK PROP}
	Let $\mathfrak{n}\colon J \to T$ and $\mathfrak{n}' \colon J' \to T'$ be necklaces. Then:
\begin{enumerate}
\item[(i)]
	A map $\mathfrak{n} \to \mathfrak{n}'$ in $\mathsf{Nec}$
	is uniquely determined by some map 
	$T \to T'$ in $\Omega$. 
	More precisely, there exists an unique dashed arrow
	making the following commute.
\begin{equation}\label{NECKMAP EQ}
\begin{tikzcd}
	\Omega[\mathfrak{n}] 
	\ar[hookrightarrow]{r} 
	\ar{d}
&
	\Omega[T] 
	\ar[dashed]{d}{\exists !}
\\
	\Omega[\mathfrak{n}']
	\ar[hookrightarrow]{r}
&
	\Omega[T']
	\end{tikzcd}
\end{equation}
\item[(ii)]
	A map of trees 
	$\varphi \colon T \to T'$ in $\Omega$
	induces a map 
	$\mathfrak{n} \to \mathfrak{n}'$ in $\mathsf{Nec}$
	iff
	$\varphi J \supseteq J'_{\overline{\varphi T}}$.
\end{enumerate}
\end{proposition}

\begin{proof}
	For $U \in \mathsf{Face}_{sc}(J)$
	the composite
	$\Omega[\mathfrak{n}] \to 
	\Omega[\mathfrak{n}'] \to 
	\Omega[T']$
	in \eqref{NECKMAP EQ}
	gives compatible maps
	$\Omega[\overline{\mathfrak{n} U}] \to \Omega[T']$
	in $\mathsf{dSet}$,
	and thus compatible maps
	$\overline{\mathfrak{n} U} \to T'$ 
	in $\Omega$,
	so (i) follows from 
	Remark \ref{GRAFT REM}.

	We now turn to (ii).
	The map $\varphi$ defines a map of necklaces 
	precisely if it induces maps
	$\Omega[T_b] \to \Omega[\mathfrak{n}']$
	for each bead 
	$T_b$, $b \in \boldsymbol{B}(\mathfrak n)$
	and, by Lemma \ref{FACEINNECK LEM},
	this is equivalent to
	\begin{equation}\label{MAPDES EQ}
	\emptyset
	=
	\boldsymbol{E}^{\mathsf{i}}(J')
	\cap
	\boldsymbol{E}^{\mathsf{i}}(\overline{ \varphi T_b})
	=
	\boldsymbol{E}^{\mathsf{i}}(J')
	\cap
	\boldsymbol{E}^{\mathsf{i}}(\overline{ \varphi \mathfrak{n} J_b}).
	\end{equation}
	
	Writing $\tilde{\varphi}$ for the composite
	$
	J \xrightarrow{\mathfrak n} T 
	\xrightarrow{\varphi} \overline{\varphi T}
	$
	and noting that $\tilde{\varphi}$ is tall,
	\eqref{EDGEBREAK EQ} becomes 
	\begin{equation}\label{DECOMPPR EQ}
	\boldsymbol{E}(\overline{\varphi T})
	=
	\boldsymbol{E}(\tilde{\varphi} J)
	\amalg
	\coprod_{b \in \boldsymbol{V}(J)}
	\boldsymbol{E}^{\mathsf{i}}(\overline{\tilde{\varphi} J_b})
	=
	\boldsymbol{E}(\tilde{\varphi} J)
	\amalg
	\coprod_{b \in \boldsymbol{B}(\mathfrak n)}
	\boldsymbol{E}^{\mathsf{i}}(\overline{\varphi T_b}).
	\end{equation}
	Thus, \eqref{MAPDES EQ}
	amounts to
	$\boldsymbol{E}^{\mathsf{i}}(J') \cap 
	\boldsymbol{E}(\overline{\varphi T})
	\subseteq
	\boldsymbol{E}(\varphi J)$,
	which is equivalent to the desired
	$\varphi J \supseteq J'_{\overline{\varphi T}}$
	(as these trees have the same outer edges).
\end{proof}

\begin{remark}\label{NECKMAPCHAR REM}
	Let $\mathfrak{n},\mathfrak{n}',T,T'$ be as in
	Proposition \ref{MAPNECK PROP}
	and suppose 
	$\varphi \colon T \to T'$
	defines a map
	$\mathfrak{n} \to \mathfrak{n}'$.
	Then for every outer face $F \to T$
	it follows from 
	Corollary \ref{NECINT COR}
	that the restriction 
	$F \to \overline{\varphi F}$
	likewise induces a restriction
	$\mathfrak{n}_F \to \mathfrak{n}'_{\overline{\varphi F}}$,
	from which it follows that
	$\varphi J_F \supseteq 
	\left(J'_{\overline{\varphi T}}\right)_{\overline{\varphi F}}
	=
	J'_{\overline{\varphi F}}$.
\end{remark}

\begin{remark}\label{BEADMAP REM}
	Let $\mathfrak{n},\mathfrak{n}',T,T',\varphi$ be as in the previous remark and suppose in addition that $\varphi$ is a face map.
	Then, since different beads share at most one edge,
	for each bead $T_{b} \hookrightarrow T$
	of $\mathfrak{n}$,
	there is a unique bead
	$T'_{\varphi_{\**} b} \hookrightarrow T'$
	of $\mathfrak{n}'$
	such that
	$T_b \hookrightarrow T \to T'$
	factors as
	$T_b \to T'_{\varphi_{\**}b} \hookrightarrow T'$.
	In particular, 
	this defines a map of bead sets
	$\varphi_{\**} \colon 
	\boldsymbol{B}(\mathfrak{n}) \to 
	\boldsymbol{B}(\mathfrak{n}')$.
\end{remark}

\section{The dendroidal $W_!$-construction}\label{WCONS SEC}

This section will establish the description
of the $W_!$-construction in \eqref{SOPDSET_EQ}
given in Theorem \ref{THMB}.

Throughout we make  
use of the factorizations in $\Omega$ given in 
Proposition \ref{TREEFACT_PROP},
and follow Notation \ref{MAPLABELS_NOT}
by labeling a map by the letters d/i/o/t/f/p
to indicate that the map is
a degeneracy/inner face/outer face/tall/face/planar.
Moreover, we implicitly use
Remark \ref{TREEFACTNAMES_REM},
stating that for some types of maps the factorization 
\eqref{TREEFACT_EQ} has only certain factors,
as well as Remark \ref{IGNPL REM},
which combines factors in \eqref{TREEFACT_EQ}
to obtain simplified factorizations.
 
We first build $W(T)$ for a tree $T$, cf. Proposition \ref{PROPA PROP}.

\begin{definition}\label{NWTNS DEF}
	Let $T \in \Omega$ be a tree.
	We define 
	$W(T) \in \mathsf{sOp}$
	to be the simplicial operad whose nerve is the 
	simplicial dendroidal set 
	$NW(T) \in \mathsf{sdSet}$
	(cf. Remark \ref{SIMPOPREM})
	with $n$-simplices given by
\begin{equation}\label{NWTDEF EQ}
	NW(T)_n(S)
=
	\left\{
	\text{composable strings }
	S \xrightarrow{t} 
	J_0 \xrightarrow{i,p} 
	J_1 \xrightarrow{i,p} 
	\cdots \xrightarrow{i,p}
	J_n \xrightarrow{i,p}
	F \xrightarrow{o,p}
	T
	\text{ of arrows in $\Omega$}
	\right\}.
\end{equation}
	Equivalently, 
	it suffices to require that the
	$S \to J_i$ are tall maps
	and the $J_i \to T$ are planar face maps.
	We note that $F$ is superfluous, 
	being determined by $J_n \xrightarrow{f,p} T$,
	but including it will make \eqref{NWNWORKS DEF} 
	below more readable.
	See also Remark \ref{J0JFREPACK REM}.

	Functoriality of 
	$NW(T)$
	with respect to a map $S^{\**} \to S$
	is described by the diagram
\begin{equation}\label{NWFUNC EQ}
\begin{tikzcd}
	S^{\**} \ar{r}{t} \ar{d}
&
	J^{\**}_0 \ar{r}{i,p} \ar{d}{o}
&
	J^{\**}_1 \ar{r}{i,p} \ar{d}{o}
&
	\cdots \ar{r}{i,p}
&
	J^{\**}_n \ar{r}{i,p} \ar{d}{o}
&
	F^{\**} \ar{r}{o,p} \ar{d}{o}
&
	T \ar[equal]{d}
\\
	S \ar{r}{t} 
&
	J_0 \ar{r}{i,p}
&
	J_1 \ar{r}{i,p}
&
	\cdots \ar{r}{i,p}
&
	J_n \ar{r}{i,p}
&
	F \ar{r}{o,p}
&
	T	
\end{tikzcd}
\end{equation}
	where the maps $J^{\**}_k \to J_k$ and 
	$F^{\**} \to F$
	are inductively defined by taking
	$S^{\**} \to J^{\**}_0 \to J_0$
	(resp. $J^{\**}_k \to J^{\**}_{k+1} \to J_{k+1}$,
	$J^{\**}_n \to F^{\**} \to F$)
	to be the ``tall followed by outer face''
	factorization of the composite
	$S^{\**} \to S \to J_0$
	(resp. $J^{\**}_k \to J_{k} \to J_{k+1}$,
	$J^{\**}_n \to J_n \to F$).

	More generally, 
	given a necklace $\mathfrak{n}\colon J \to T$,
	we define
	$NW(\mathfrak{n}) \subseteq NW(T)$
	as the subpresheaf formed by those strings in 
	\eqref{NWTDEF EQ} such that
	one has $J_0 \supseteq J_{F}$
	(where $J_{F}$ is as in Notation \ref{JFNOT NOT}).
	The fact that this $NW(\mathfrak{n})$ is a presheaf follows 
	since, for $S^{\**} \to S$, 
	$J_1,J_i^{\**}$, $F,F^{\**}$ as in \eqref{NWFUNC EQ},
	it is
\begin{equation}\label{NWNWORKS DEF}
	\boldsymbol{E}^{\mathsf{i}}(J^{\**}_0)
=
	\boldsymbol{E}^{\mathsf{i}}(J_0)
	\cap
	\boldsymbol{E}^{\mathsf{i}}(F^{\**})
\supseteq
	\boldsymbol{E}^{\mathsf{i}}(J_{F})
	\cap
	\boldsymbol{E}^{\mathsf{i}}(F^{\**})
=
	\boldsymbol{E}^{\mathsf{i}}(J_{F^{\**}})
\end{equation}
where the first step is 
\cite[Lemma 2.5]{BP_edss}
applied to 
$J^{\**}_0 \xrightarrow{o} J_0 \xrightarrow{i} F$
and 
$J^{\**}_0 \xrightarrow{i} 
F^{\**} \xrightarrow{o}
F$,
the second is the definition of $NW(\mathfrak{n})$,
and the third follows from Notation \ref{JFNOT NOT}
and the fact that $F^{\**} \subseteq F$.
\end{definition}

\begin{remark}\label{J0JFREPACK REM}
	Writing $\phi \colon S \to T$ for the full composite in \eqref{NWTDEF EQ},
	one has that the $F$ therein is $\overline{\phi S}$
	(cf. Notation \ref{TREEFACT NOT})
	In particular, the condition $J_0 \supseteq J_F$
	defining $NW(\mathfrak{n})$
	becomes $J_0 \supseteq J_{\overline{\phi S}}$.
\end{remark}

\begin{remark}\label{NWTNS_REM}
	The $NW(\mathfrak{n})$ given by Definition \ref{NWTNS DEF}
	are nerves of simplicial operads,
	cf. Remark \ref{SIMPOPREM}.
	Indeed, to verify the Segal condition \eqref{SCV2 EQ} note that,
	as the maps $S \to J_i$ in \eqref{NWFUNC EQ}
	are tall, they are uniquely determined by maps
	$S_v \to J_{i,v}$ for $v \in \boldsymbol{V}(S)$,
	cf. \cite[Cor 3.75]{BP_geo} 
	(see also \ref{GRAFT REM}).
	Moreover,  $NW(\mathfrak{n})(\eta)$ is a discrete simplicial set
	since for $S=\eta$ it must be $J_i = \eta$ in \eqref{NWTDEF EQ},
	due to only $\eta$ receiving tall maps from $\eta$.
\end{remark}

Next, we discuss the functoriality of
$NW(T)$ with respect to $T \in \Omega$.
For a map $T \to T'$ in $\Omega$ we define
$NW(T)_n(S)
	\to 
NW(T')_n(S)$
via the diagram (where we drop the superfluous $F$ in \eqref{NWTDEF EQ})
\begin{equation}\label{NWTTSTAR EQ}
\begin{tikzcd}
	S \ar{r}{t} \ar[equal]{d}
&
	J_0 \ar{r}{i,p} \ar{d}{d}
&
	J_1 \ar{r}{i,p} \ar{d}{d}
&
	\cdots \ar{r}{i,p}
&
J_n \ar{r}{f,p} \ar{d}{d}
&
	T \ar{d}
\\
	S \ar{r}{t} 
&
	J'_0 \ar{r}{i,p}
&
	J'_1 \ar{r}{i,p}
&
	\cdots \ar{r}{i,p}
&
	J'_n \ar{r}{f,p}
&
	T'
\end{tikzcd}
\end{equation}
where the maps $J_k \to J'_k$
are (backwards) inductively defined
by taking 
$J_n \to J'_n \to T'$
(resp. 
$J_{k-1} \to J'_{k-1} \to J'_k$)
to be the "degeneracy followed by face"
factorization of the composite
$J_n \to T \to T'$
(resp.
$J_{k-1} \to J_{k} \to J'_k$).

\begin{proposition}\label{NWTNS_NAT_PROP}
	For any map $T \to T'$ in $\Omega$, the induced map
	$NW(T)(S)
	\to 
	NW(T')(S)$
	in 
	\eqref{NWTTSTAR EQ}
	is functorial on $S$.
\end{proposition}

\begin{proof}
	First, note that the composite
	$NW(T)({S})
	\to 
	NW(T)({S^{\**}})
	\to 
	NW(T')({S^{\**}})$
	is computed by the left diagram below,
	where
	$S^{\**} \to J^{\**}_i \to J_i$
	and 
	$J^{\**}_i \to \left(J_i^{\**}\right)' \to T'$
	are the unique factorizations with the indicated properties.
	On the other hand, the composite
	$NW(T)(S)
	\to 
	NW(T')(S)
	\to 
	NW(T')(S^{\**})$
	is computed as on the right
	with 
	$J_i \to J'_i \to T'$ and
	$S^{\**} \to \left(J_i'\right)^{\**} \to J'_i$
	the unique indicated factorizations.
\begin{equation}\label{NWTNS_NAT_EQ1}
\begin{tikzcd}
	S \ar{r}{t} 
&
	J_i \ar{r}{f,p} 
&
	T \ar[equal]{d}
&&
	S \ar{r}{t} \ar[equal]{d}
&
	J_i \ar{r}{f,p} \ar{d}{d}
&
	T \ar{d}
\\
	S^{\**} \ar{u} \ar{r}{t} \ar[equal]{d}
&
	J^{\**}_i \ar{r} \arrow{u}[swap]{o,p} \arrow{d}{d}
&
	T \ar{d}
&&
	S \ar{r}
&
	J'_i \arrow{r}{f,p}
&
	T' \ar[equal]{d}
\\
	S^{\**} \ar{r}
&
	\left(J_i^{\**}\right)' \arrow{r}{f,p}
&
	T'
&&
	S^{\**} \arrow{r}{t} \ar{u}
&
	\left(J_i'\right)^{\**} \ar{r} \arrow{u}[swap]{o,p}
&
	T'
\end{tikzcd}
\end{equation}
	The key to the proof is to show
	that the planar faces 
	$\left(J_i^{\**}\right)'$ and
	$\left(J_i'\right)^{\**}$
	of $T'$
	coincide, since it will then be automatic that all maps connecting the
	$\left(J_i^{\**}\right)'$ and
	$\left(J_i'\right)^{\**}$
	and $S^{\**}$, $T'$
	likewise match.

	To see this, we consider the following diagram
	which combines the top halves in \eqref{NWTNS_NAT_EQ1}.
\begin{equation}
\begin{tikzcd}
	S^{\**} \ar{r}{t} \ar{d}
&
	J^{\**}_i \ar{r} \ar{d}{o,p}
&
	T \ar[equal]{d}
\\
	S  \ar{r}{t} \ar[equal]{d}
&
	J_i \ar{r}{f,p}  \ar{d}{d}
&
	T \ar{d}
\\
	S \ar{r}
&
	J_i' \ar{r}{f,p}
&
	T'
\end{tikzcd}
\end{equation}
	Both faces  
	$\left(J^{\**}_i\right)'$ and 
	$\left(J'_i\right)^{\**}
	$
	can be built by factoring
	the composite 
	$J^{\**}_i \to J_i \to J_i'$,
	with 
	$\left(J^{\**}_i\right)'$ 
	coming from the 
	degeneracy-face factorization
	and 
	$\left(J'_i\right)^{\**}$
	coming from the 
	tall-outer factorization.
	But since 
	$J^{\**}_i \to J_i \to J_i'$
	is a composite of convex maps (cf. Definition \ref{TREEMAP_DEF})
    it is again convex (see Remark \ref{TREEFACTNAMES_REM}), 
	so the two factorizations coincide, 
	finishing the proof.
\end{proof}

\begin{corollary}\label{NWTNS_NNAT_COR}
	Let $\mathfrak{n} \colon J \to T$ and
	$\mathfrak{n}' \colon J' \to T'$
	be necklaces and suppose 
	$\psi \colon T\to T'$
	induces a map $\mathfrak{n} \to \mathfrak{n}'$.
	Then the induced map
	$NW(T) \to NW(T')$
	restricts to a map
	$NW(\mathfrak{n}) \to NW(\mathfrak{n}')$.
\end{corollary}

\begin{proof}
	Following Remark \ref{J0JFREPACK REM},
	we need to show that the map
	$NW(T) \to NW(T')$
	sends simplices \eqref{NWTDEF EQ}
	such that
	$J_0 \supseteq 
	J_{\overline{\phi S}}$
	to simplices such that
	$J'_0 \supseteq 
	J'_{\overline{\phi' S}}$,
	where $\phi,\phi'$ are the composites of each simplex.
	This follows since
\[
	J'_0 = 
	\psi (J_0) \supseteq
	\psi (J_{\overline{\phi S}})
	\supseteq
	J'_{\overline{\psi (\overline{\phi S})}}
	=
	J'_{\overline{\psi' S}}
\]	
where the third step is
Remark \ref{NECKMAPCHAR REM}.
\end{proof}

We now introduce a notation that plays an important role 
in two key technical results,
Propositions \ref{NECKCOL PROP} and 
\ref{NWKANEX_PROP}.
Recall that, for any tree 
$U \in \Omega$, 
the poset $\mathsf{Face}_{inn}(U)$
of planar inner faces is in fact a lattice,
with the join $F \vee F'$
the characterized by
$\boldsymbol{E}^{\mathsf{i}}(F \vee F') 
=
\boldsymbol{E}^{\mathsf{i}}(F)
\cup
\boldsymbol{E}^{\mathsf{i}}(F')$.

\begin{notation}\label{STAU NOT}
	Let $\mathfrak{n} \colon J \to T$
	be a necklace,
	$\phi\colon S \to T$ a map in $\Omega$
	and $S \xrightarrow{t} F \xrightarrow{p,o} T$
	its tall-outer factorization.
	We then write
	$J^{\phi} = \phi S \vee J_{F}$,
	where the join is 
	taken in $\mathsf{Face}_{inn}(F)$.
\end{notation}

\begin{remark}
	Following Notation \ref{MAPSOVCOL NOT},
	we write $NW(\mathfrak{n})_{\phi}(S) \subseteq NW(\mathfrak{n})(S) \in \sSet$
	for the subsimplicial set 
	over a coloring $\phi \colon \boldsymbol{E}(S) \to \boldsymbol{E}(T)$.
	Note that,
	by the description in \eqref{NWTDEF EQ},
	$\phi$ must in fact be the map of trees $\phi\colon S \to T$
	given by the composite therein.
\end{remark}

\begin{remark}\label{STAU REM}
	In the context of Notation \ref{STAU NOT},
	and writing 
	$\bar{\phi} \colon S \to F$,
	$\iota \colon F \to T$,
	$\iota^{\phi} \colon J^{\phi} \to T$,
	for the natural maps,
	one has identifications
	\begin{equation}
	\begin{tikzcd}[column sep = 12pt]
	NW(\mathfrak{n})_{\iota^{\phi}}(J^{\phi}) 
	\ar{r}{\simeq}
	&
	NW(\mathfrak{n})_{\phi}(S)
	&
	NW(\mathfrak{n}_F)_{\bar{\phi}}(S)
	\ar{l}[swap]{\simeq}
	&
	NW(J^{\phi} \to F)_{\bar{\phi}}(S)
	\ar{l}[swap]{\simeq}
	\end{tikzcd}
	\end{equation}
	induced by the natural maps
	$S \to J^{\phi}$ between trees and 
	$(J^{\phi} \to F) \to
	\mathfrak{n}_F \to \mathfrak{n}$
	between necklaces.
\end{remark}

\begin{proposition}\label{NECKCOL PROP}
	Let $\mathfrak{n} \colon J \to T$ be a necklace.
	Then one has an identification
	\begin{equation}\label{NECKCOL EQ}
	W(\mathfrak{n})
	\simeq 
	\underset{U \in \mathsf{Face}_{sc}(J)}{\colim}
	W(T_U)
	\end{equation}
	where the colimit takes place 
	in $\mathsf{sOp}$.
\end{proposition}

\begin{proof}
	We will verify \eqref{NECKCOL EQ}
	at the level of nerves.
	More explicitly, 
	we will show that for any $X \in \mathsf{sdSet}$
	with constant objects and 
	satisfying the strict Segal condition
	(cf. Remark \ref{SIMPOPREM}),
	giving a map
	$NW(\mathfrak{n}) \to X$
	is the same as giving compatible maps
	$NW(T_U) \to X$.
	
	Moreover, clearly both sides of 
	\eqref{NECKCOL EQ} yield $\boldsymbol{E}(T)$
	when evaluated at $\eta$.
	As such, we are free to fix a coloring
	$\mathfrak{c} \colon \boldsymbol{E}(T) \to X(\eta)$
	and verify the universal property
	restricted to maps respecting this color assignment.
	And, by using the identification \eqref{DSETCID EQ} and the 
	last comment in Remark \ref{PULLBACKS REM},
	we may evaluate $NW(\mathfrak{n}),X$
	on $\boldsymbol{E}(T)$-colored trees $(S,\phi)$, 
	rather than on uncolored trees.

	Given maps $NW(T_U) \to X$
	we now define the map $NW(\mathfrak{n}) \to X$ via 
	(where $J^{\phi}$, $F$, $\bar{\phi}$, $\iota$, $\iota^{\phi}$
	are as in Remark \ref{STAU REM},
	and 
	$\iota_{\**} \colon 
	\boldsymbol{B}(\mathfrak{n}_{F})
	\to 
	\boldsymbol{B}(\mathfrak{n})$
	is the map of bead sets in Remark \ref{BEADMAP REM})
\begin{equation}\label{LONGMAP EQ}
\begin{tikzcd}[column sep = 20pt, row sep=0]
	NW(\mathfrak{n})_{\phi}(S)
&
	NW(\mathfrak{n})_{\iota^{\phi}}(J^{\phi})
	\ar{l}[swap]{\simeq}{(I)}
	\ar{r}{\simeq}[swap]{(II)}
&
	\underset{b \in \boldsymbol{B}(\mathfrak{n}_{F})}{\prod}
	NW(\mathfrak{n})_{\iota^{\phi}_b}(J^{\phi}_b)
\\
&&&
	\underset{b \in \boldsymbol{B}(\mathfrak{n}_{F})}{\prod}
	NW(T_{\iota_{\**}b})_{\iota^{\phi}_b}(J^{\phi}_b)
	\ar{lu}[swap]{\simeq}{(III)}
	\ar{ld}{(IV)}
\\
	X_{\mathfrak{c}\phi}(S)
&
	X_{\mathfrak{c}\iota^{\phi}}(J^{\phi})
	\ar{r}{\simeq}[swap]{(V)}
	\ar{l}{(VI)}
&
	\underset{b \in \boldsymbol{B}(\mathfrak{n}_{F})}{\prod}
	X_{\mathfrak{c}\iota^{\phi}_b}(J^{\phi}_b)
\end{tikzcd}
\end{equation}
	where the arrows (II) and (V) are isomorphisms by the Segal condition
	while (III) is an isomorphism since 
	$J^{\phi}_b \to T$ factors through 
	$T_{\iota_{\**}b}$.
	
	Moreover, the arrow (IV) in \eqref{LONGMAP EQ}
	is induced by the chosen maps
	$NW(T_U) \to X$,
	so clearly 
	\eqref{LONGMAP EQ}
	denotes the only possible compatible map
	$NW(T) \to X$.

	It only remains to check that
	\eqref{LONGMAP EQ}
	is indeed a map in 
	$\mathsf{sdSet}$, i.e. that it is natural on 
	$(S,\phi)$.
	To see this, one first readily checks that a map
	$\psi \colon (S,\phi) \to (S^{\**},\phi^{\**})$
	induces a compatible inclusion
	$\psi \colon J^{\phi} \hookrightarrow J^{\phi^{\**}}$
	showing the naturality of arrows (I),(VI) 
	in the zigzag.
	Next, by Remark \ref{BEADMAP REM}
	one has a map of bead sets
	$\psi_{\**} \colon 
	\boldsymbol{B}(\mathfrak{n}_{F})
	\to
	\boldsymbol{B}(\mathfrak{n}_{F^{\**}})$
	for which one has further compatible maps
	$J^{\phi}_b
	\to 
	J^{\phi^{\**}}_{\psi_{\**}b}$,
	showing the naturality of the arrows (II),(V).
	Lastly, for any bead 
	$b \in \boldsymbol{B}(\mathfrak{n}_{F})$
	one has
	$T_{\iota_{\**}b} = T_{(\iota^{\**})_{\**} \psi_{\**}b}$,
	showing naturality of the arrows (III),(IV).
\end{proof}

\begin{remark}\label{PREOPCOLEV REM}
	Let $I \xrightarrow{A_{\bullet}} \mathsf{dSet}$
	be a diagram of dendroidal sets and let
	$A = \colim_{i \in I} A_i$.

	We will find it useful to describe $A$
	in light of the identification \eqref{DSETCID EQ}.
	For each $A(\eta)$-colored tree 
	$\vect{S} = (S,\mathfrak{c})$
	we write
	$I_{\vect{S}/}$
	for the category with objects factorizations
	$\boldsymbol{E}(S) \to A_i(\eta) \to A(\eta)$ 
	for some $i \in I$,
	which we represent by 
	$\boldsymbol{E}(S) \to A_i(\eta)$,
	together with maps $i \to i'$ in $I$
	satisfying the obvious compatibility.
	Then
\begin{equation}\label{COLIMALT EQ}
	A_{\mathfrak{c}}(S) \simeq 
	\underset{(\boldsymbol{E}(S) \to A_i(\eta)) \in I_{\vect{S}/}}{\colim}
	A_{i,
	\mathfrak{c}_i}(S)
\end{equation}
	where $\mathfrak{c}_i$ in 
	$A_{i,\mathfrak{c}_i}(S)$
	denotes the coloring given by
	$\boldsymbol{E}(S) \to A_i(\eta)$.
\end{remark}

\begin{proposition}\label{NWKANEX_PROP}
	Let $X \in \mathsf{dSet}$ and define
	$NW(X) \in \mathsf{sdSet}$ by
	\begin{equation}\label{NWKANEX EQ}
	NW(X) =
	\underset{(\Omega[\mathfrak{n}] \to X)
		\in \mathsf{Nec}_{/X}}{\colim}
	NW(\mathfrak{n}).
	\end{equation}
	where
	$\mathsf{Nec}_{/X} = \mathsf{Nec} \downarrow X$ is the over category of maps $\Omega[\mathfrak{n}] \to X$, and
	the colimit is taken in $\mathsf{sdSet}$.

	Then $NW(X)$ satisfies the strict Segal condition
	\eqref{SCV1 EQ},\eqref{SCV2 EQ},
	and has constant objects, cf. Remark \ref{SIMPOPREM}.
	In particular, since $N$ is fully-faithful one has that
	$NW(X)$ is the nerve of the simplicial operad
\[
	W(X) =
	\underset{
	(\Omega[\mathfrak{n}] \to X)
	\in \mathsf{Nec}_{/X}}{\colim}
	W(\mathfrak{n}).
\]
	where the colimit is now taken in simplicial operads
	$\mathsf{sOp}$.
\end{proposition}

\begin{proof}
	We will evaluate 
	$NW(X)$ at each $X(\eta)$-colored tree
	$\vect{S}=(S,\mathfrak{c})$
	using Remark \ref{PREOPCOLEV REM}.
	We write
	$\mathsf{Nec}_{\vect{S}//X}
	=
	\left(\mathsf{Nec}_{/X}\right)_{\vect{S}/}$
	for the category
	whose objects are pairs of arrows
	$\boldsymbol{E}(S)
	\xrightarrow{\phi_{\mathfrak{n}}} \Omega[\mathfrak{n}] \to X$
	whose composite encodes the coloring 
	$\mathfrak{c}\colon \boldsymbol{E}(T) \to X(\eta)$.
	Equation \eqref{COLIMALT EQ}
	then says that
\begin{equation}\label{NWKANEXEV EQ}
	NW(X)_{\mathfrak{c}}(S) 
	\simeq
	\underset{(\boldsymbol{E}(S) \to \Omega[\mathfrak{n}] \to X)
		\in \mathsf{Nec}_{\vect{S}//X}}{\colim}
	NW(\mathfrak{n})_{\phi_{\mathfrak{n}}}(S).
\end{equation}

	To show that $NW(X)$ satisfies the strict Segal condition, 
	we will rewrite \eqref{NWKANEXEV EQ} 
	by identifying appropriate subcategories of
	$\mathsf{Nec}_{\vect{S}//X}$.
	First, write
	$\mathsf{Nec}_{\vect{S}//X}^{\Omega}
	\subset
	\mathsf{Nec}_{\vect{S}//X}$
	for the full subcategory of those objects for which,
	writing $\mathfrak{n}\colon J \to T$,
	the map
	$\phi_{\mathfrak{n}} \colon 
	\boldsymbol{E}(S) \to \boldsymbol{E}(T)$
	gives a map 
	$\phi_{\mathfrak{n}} \colon S \to T$ in $\Omega$.

	Next, for $\phi_{\mathfrak{n}} \colon S \to T$ as above, 
	and as in Notation \ref{STAU NOT},
	we write
	$S \xrightarrow{\bar{\phi}_{\mathfrak n}} 
	F_{\mathfrak{n}} \xrightarrow{\iota} T$
	for the tall-outer factorization.
	We then write 
	$
	\mathsf{Nec}_{\vect{S}//X}^{\Omega,nor}
	\subset
	\mathsf{Nec}_{\vect{S}//X}^{\Omega}
	$
	for the full subcategory of ``normalized factorizations'',
	defined by the properties that
	$\phi_{\mathfrak{n}} \colon S \to T$ is a tall map,
	i.e. $F_{\mathfrak n}  = T$, and
	$J \supseteq \phi_{\mathfrak{n}} S$.

	Moreover, there is a retraction
	$
	\mathsf{Nec}_{\vect{S}//X}^{\Omega}
	\xrightarrow{n}
	\mathsf{Nec}_{\vect{S}//X}^{\Omega,nor}
	$
	which sends
	$\boldsymbol{E}(S) \Omega[\eta]
	\to \Omega[J \xrightarrow{\mathfrak{n}} T] \to X$
	to
	$n (\mathfrak{n}) = 
	(\phi_{\mathfrak{n}} S \vee J_{F_{\mathfrak{n}}}
	\to F_{\mathfrak{n}})
	=
	(J^{\phi_{\mathfrak{n}}} \to  F_{\mathfrak{n}})
	$
	(cf. Notation \ref{STAU NOT}).
	Recall (cf. Remark \ref{STAU REM}) 
	that the natural map
	$n( \mathfrak{n}) \to \mathfrak{n}$ in $\mathsf{Nec}$
	induces isomorphisms
	$NW(n(\mathfrak{n}))_{\bar{\phi}_{\mathfrak{n}}}(S) 
	\xrightarrow{\simeq}
	NW(\mathfrak{n})_{\phi_{\mathfrak{n}}}(S)$.

	Since 
	$\mathsf{Nec}_{\vect{S}//X}^{\Omega}$
	is a cosieve\footnote{
	Recall that a subcategory 
	$\mathcal{S} \subseteq \mathcal{C}$
	is a \emph{cosieve} if, for any map
	$s \to c$ with $s \in \mathcal{S}$,
	both $c$ and $s \to c$ are also in $\mathcal{S}$.} 
	of $\mathsf{Nec}_{\vect{S}//X}$
	and 
	$NW(\mathfrak{n})_{\mathfrak{s}}(S) = \emptyset$
	whenever $\mathfrak{n}$ is not in 
	$\mathsf{Nec}_{\vect{S}//X}^{\Omega}$
	one can replace 
	$\mathsf{Nec}_{\vect{S}//X}$
	with 
	$\mathsf{Nec}_{\vect{S}//X}^{\Omega}$
	in
	\eqref{NWKANEXEV EQ}.
	Moreover, the existence of a retraction implies that the inclusion
	$
	\mathsf{Nec}_{\vect{S}//X}^{\Omega,nor}
	\subset
	\mathsf{Nec}_{\vect{S}//X}^{\Omega}
	$
	is initial, so one can further replace 
	$\mathsf{Nec}_{\vect{S}//X}^{\Omega}$ with
	$\mathsf{Nec}_{\vect{S}//X}^{\Omega,nor}$.
	
	Lastly, note that the normalization conditions
	imply that $\mathsf{Nec}_{\vect{S}//X}^{\Omega,nor} 
	\simeq
	\prod_{v \in \boldsymbol{V}(S)}
	\mathsf{Nec}_{\vect{S_v}//X}^{\Omega,nor}$
	by a grafting argument.
	Putting everything together we now obtain that
\begin{equation}\label{NWKANEXEVEXT EQ}
\begin{aligned}
	NW(X)_{\mathfrak{c}}(S) 
\simeq &
	\underset{(\boldsymbol{E}(S) \to \Omega[\mathfrak{n}] \to X)
		\in \mathsf{Nec}^{\Omega,nor}_{\vect{S}//X}}{\colim}
	NW(\mathfrak{n})_{\phi_\mathfrak{n}}(S)
\\
\simeq &
\underset{(\boldsymbol{E}(S_v) \to \Omega[\mathfrak{n}_{T_v}] \to X)
	\in 
	\prod_{v \in \boldsymbol{V}(S)} \mathsf{Nec}^{\Omega,nor}_{\vect{S_v}//X}}{\colim}
	\left(\prod_{v \in \boldsymbol{V}(S)}
	NW(\mathfrak{n})_{\phi_{\mathfrak{n},v}}(S_v)\right)
\\
\simeq &
	\underset{(\boldsymbol{E}(S_v) \to \Omega[\mathfrak{n}_{T_v}] \to X)
		\in 
		\prod_{v \in \boldsymbol{V}(S)} \mathsf{Nec}^{\Omega,nor}_{\vect{S_v}//X}}{\colim}
	\left(\prod_{v \in \boldsymbol{V}(S)}
	NW(\mathfrak{n}_{T_v})_{\phi_\mathfrak{n},v}(S_v)\right)
\\
\simeq &
	\prod_{v \in \boldsymbol{V}(S)}
	\left(
	\underset{(\boldsymbol{E}(S_v) \to \mathfrak{n}_{T_v} \to X)
		\in \mathsf{Nec}^{\Omega,nor}_{\vect{S_v}//X}}{\colim}
	NW(\mathfrak{n}_{T_v})_{\phi_{\mathfrak{n},v}}(S_v)
	\right)
\\
\simeq &
	\prod_{v \in \boldsymbol{V}(S)} 
	NW(X)_{\mathfrak{c}_v}(S_v) 
\end{aligned}
\end{equation}
	where the first step follows from the previous paragraph,
	the second step is the identification of indexing categories
	above together with the strict Segal condition for 
	$NW(\mathfrak{n})$,
	the third step uses the middle isomorphisms in 
	Remark \ref{STAU NOT},	
	the fourth step is the fact that products commute with colimits in each variable, and the last step simply specifies
	the first step for $\vect{S_v} = (S_v,\mathfrak{c}_v)$.
	
	We have thus established the 
	strict Segal condition for $NW(X)$ so that,
	as it is clear that the objects of $NW(X)$ are discrete 
	(this is inherited from the $NW(\mathfrak{n})$),
	this finishes the proof.
\end{proof}

\begin{remark}
	The normalization condition in the previous proof
	is equivalent to requiring that 
	$\phi_{\mathfrak{n}} \colon S \to T$
	is a tall map which induces a map
	$\mathsf{Sc}[S] \to \Omega[\mathfrak{n}]$.
\end{remark}

Propositions \ref{NECKCOL PROP} and 
\ref{NWKANEX_PROP}
now combine to give the following, establishing \eqref{KEYSTRAT EQ}.

\begin{theorem}[{cf. \cite[Thm. 1.3]{DS11}}]
	\label{KANEXTCHAR THM}
	Consider the following diagram,
	where the functors labeled $W$ 
	are as defined by Definition \ref{NWTNS DEF}
	and Proposition \ref{NWKANEX_PROP}.
\begin{equation}
\begin{tikzcd}
	\Omega \ar[hookrightarrow]{r}
	\ar{rrd}[swap]{W} 
&
	\mathsf{Nec}
	\ar[hookrightarrow]{r}
	\ar{rd}{W}
&
	\mathsf{dSet}
	\ar{d}{W}
\\
&&
	\mathsf{sOp} 
\end{tikzcd}
\end{equation}
Then both triangles in this diagram are left Kan extensions.
In particular, the functor
$W \colon \mathsf{dSet} \to \mathsf{sOp}$
coincides with the functor
$W_! \colon \mathsf{dSet} \to \mathsf{sOp}$
as defined in \eqref{SOPDSET_EQ}.
\end{theorem}

\begin{remark}\label{TAUFUNEX REM}
	Our arguments so far can readily be modified 
	to also describe $\tau$ in the 
	$\tau \colon \dSet \rightleftarrows \Op \colon N$
	adjunction in \eqref{SOPDSET_EQ}.
	First, Remark \ref{SEGFORNECK REM} implies that,
	for any necklace 
	$\mathfrak{n} \colon J \to T$ it must be 
	$\tau \Omega[\mathfrak{n}] \simeq
	\tau \Omega[T] = \Omega(T)$,
	i.e. the operad such that $N\Omega(T) = \Omega[T]$,
	cf. Example \ref{OMTSEG EX}.
	We note that, in light of Remark \ref{GRAFT REM},
	this last observation is the analogue of 
	Proposition \ref{NECKCOL EQ} for $\tau$.
	
	Adapting the proofs of 
	Proposition \ref{NWKANEX_PROP} and
	Theorem \ref{KANEXTCHAR THM},
	and writing $\mathfrak{c},\phi$ as therein, one then has
\begin{equation}
\begin{aligned}
	\left(N \tau X\right)_{\mathfrak{c}}(S) 
	\simeq &
\underset{(\boldsymbol{E}(S) \to 
	\Omega[J \to T] \to X)
	\in \mathsf{Nec}_{\vect{S}//X}}{\colim}
	\Omega[T]_{\phi}(S)
	\simeq &
\underset{(\boldsymbol{E}(S) \to 
	\Omega[J \to T] \to X)
	\in \mathsf{Nec}^{\Omega,nor}_{\vect{S}//X}}{\colim}
	\Omega[T]_{\phi}(S).
\end{aligned}
\end{equation}
Moreover, the normalization conditions guarantee
that one tautologically has  
$\Omega[T]_{\phi}(S) = \**$
in the rightmost formula
(as $\Omega[T]_{\phi}(S)$ consists of elements lifting the prescribed $\phi\colon S \to T$).
Thus, by unpacking the left formula and specifying to the case of
$S=C$ a corolla,
one has that operations in
$\tau X(C)$
are represented by 
data of the form
$
	\Omega[C]
	\xrightarrow{t}
	\Omega[T]
	\leftarrow
	\Omega[J \to T]
	\to X
$
subject to the equivalence relation
generated by deeming two such data to be equivalent whenever there exists a map of necklaces
$(J \to T) \to (J' \to T')$
making the diagram below commute.
\[
\begin{tikzcd}
	\Omega[C] \ar{r}{t} \ar[equal]{d} &
	\Omega[T] \ar{d} &
	\Omega[J \to T] \ar{l} \ar{r} \ar{d} &
	X \ar[equal]{d}
\\
	\Omega[C] \ar{r}{t} &
	\Omega[T'] &
	\Omega[J' \to T'] \ar{l} \ar{r} &
	X
\end{tikzcd}
\]
\end{remark}

\begin{remark}\label{GTAUFUNEX REM}
	The work in this paper can be adapted to the categories
	$\mathsf{dSet}_G$ and $\mathsf{Op}_G$
	of \textit{genuine equivariant dendroidal sets} and \textit{genuine equivariant operads},
	introduced in \cite[\S 5.4]{Per18} and \cite{BP_geo}, respectively.
	In particular, 
	the ``genuine operadification'' functor
	$\tau_G \colon \mathsf{dSet}_G \to \mathsf{Op}_G$
	from
	\cite[ \eqref{TAS-TAUFUNCTS EQ}]{BP_TAS}
	can be described via an analogue of
	Remark \ref{TAUFUNEX REM}.

Briefly, $G$-trees are defined as follows. 
First, the category $\Phi$ of forests has objects
formal coproducts $\amalg_{i\in I} T_i$ of trees,
with maps $\amalg_{i\in I} T_i \to \amalg_{j\in J} S_j$
given by a map of sets $\varphi \colon I \to J$ together with maps
$T_i \to S_{\varphi (i)}$ of trees.
Then, $G$-forests $\Phi^G$ are defined as the $G$-objects in $\Phi$,
while $G$-trees $\Omega_G \subset \Phi^G$
are the full subcategory of $G$-forests for which
the $G$-action is transitive on tree components.
Extending Proposition \ref{TREEFACT_PROP},
maps in $\Omega_G$ likewise have a standard factorization
\cite[Cor. \ref{TAS-OMGFACT COR}]{BP_TAS},
which allows for a generalization of the work herein.

For instance, an \emph{equivariant necklace}
is a map $\mathfrak{n} \colon J \to T$ of $G$-trees
that is a planar orbital inner face 
\cite[Def. \ref{TAS-DENDNECK DEF}]{BP_TAS}.
Explicitly, this means that $\mathfrak{n}$
is an ordered isomorphism on roots/components which is a planar inner face on each tree component.
Moreover, letting $\Omega[T] \in \dSet^G$ for $T \in \Omega_G$
 be the representables 
in \cite[\S\ref{TAS-EDS_SEC}]{BP_TAS},
one may define $\Omega[J \to T]$ just as in 
\eqref{NECKLACE_DEF}.
Altogether, adapting Remark \ref{TAUFUNEX REM},
one has that for each $G$-corolla $C$ (i.e. $G$-tree whose tree components are corollas)
the operations in 
$\tau_G X(C) \in \Op_G$
can be represented by data
$\Omega[C] \xrightarrow{t,r}
\Omega[T] \leftarrow
\Omega[J \to T] \to 
X$
(where the map labeled $t,r$ induces an ordered isomorphism on roots which is tall in each component)
subject to the equivalence relation generated by diagrams
\[
\begin{tikzcd}
\Omega[C] \ar{r}{t,r} \ar[equal]{d} &
\Omega[T] \ar{d} &
\Omega[J \to T] \ar{l} \ar{r} \ar{d} &
X \ar[equal]{d}
\\
\Omega[C] \ar{r}{t,r} &
\Omega[T'] &
\Omega[J' \to T'] \ar{l} \ar{r} &
X.
\end{tikzcd}
\]
\end{remark}

Theorem \ref{KANEXTCHAR THM} established that 
$W_!\colon \dSet \to \sOp$ is computed by \eqref{NWKANEX EQ},
which is the hard technical step in establishing Theorem \ref{THMB}.
Thus, the remainder of the paper will mostly unpack \eqref{NWKANEX EQ}
to obtain the description of $NW(X)$ for $X \in \dSet$
featured in Theorem \ref{THMB},
with the following establishing the non-unique
description,
reformulated using the spaces $X_{\mathfrak{c}}(T)$ in
Notation \ref{MAPSOVCOL NOT}.

\begin{corollary}[{cf. \cite[Cor. 4.4]{DS11}}]
	\label{NWXREPS COR}
	Let $X\in \mathsf{dSet}$.
	Then the simplices in
	$NW(X)_{n,\mathfrak{c}}(S)$
	for a coloring 
	$\mathfrak{c}\colon \boldsymbol{E}(T) \to X(\eta)$
	are equivalence classes of quadruples
	$(\mathfrak{n}, S \xrightarrow{\phi} T, \Omega[\mathfrak{n}] \xrightarrow{x} X, J_{\bullet})$ 
	where:
	\begin{enumerate}[label=(\roman*)]
		\item $(J \xrightarrow{\mathfrak{n}} T) \in \mathsf{Nec}$ is a necklace; 
		\item $S \xrightarrow{\phi} T$
		is a tall map in $\Omega$
		such that $J \supseteq \phi S$
		(equivalently, $\phi$ induces a map
		$\mathsf{Sc}[S] \to \Omega[\mathfrak{n}]$);		
		\item $\Omega[ \mathfrak{n}] \to X$ is a map in $\mathsf{dSet}$
		such that the induced composite
		$\boldsymbol{E}(S) \to 
		\boldsymbol{E}(T) \to X(\eta)$
		is the coloring $\mathfrak{c}$;
		\item $J_{\bullet}$ denotes a simplex in $NW(\mathfrak{n})_{n,\phi}$, i.e.
		a factorization of $\phi$	
	\begin{equation}\label{STRINCOR EQ}
	\begin{tikzcd}
		S \ar{r}{t}
	&
		J_0 \ar{r}{i,p}
	&
		J_1 \ar{r}{i,p}
	&
		\cdots
		\ar{r}{i,p}
	&
		J_n \ar{r}{i,p}
	&
		T
	\end{tikzcd}
	\end{equation}
		such that 
		$J_0 \supseteq J$.
	\end{enumerate}
The equivalence relation is generated by considering 
$(\mathfrak{n},\phi,x,J_{\bullet})$ and
$(\mathfrak{n}',\phi',x',J_{\bullet}')$
to be equivalent if there is
a map
$\varphi \colon \Omega[\mathfrak{n}] \to \Omega[\mathfrak{n}']$
such that
$\phi' = \varphi \phi$,
$x = x' \varphi $
and
$J'_k = \varphi J_k$
(i.e $J'_{\bullet}$
is obtained by pushing 
$J_{\bullet}$ along $\varphi$
in the sense of \eqref{NWTTSTAR EQ}).
\end{corollary}

\begin{proof}
	Conditions (i),(iii),(iv) follow by simply unpacking \eqref{NWKANEXEV EQ} in light of the construction of 
	$NW(\mathfrak{n})$ in Definition \ref{NWTNS DEF}
	(except with $\phi$ then just a map
	$\phi\colon \boldsymbol{E}(S) \to \boldsymbol{E}(S)$
	and the last map in (iv) only required to be tall rather than inner).
	The additional condition (ii) follows by replacing 
	\eqref{NWKANEXEV EQ} with its reduction to 
	``normalized
	factorizations''
	$\mathsf{Nec}_{\vec{S}//X}^{\Omega,nor}$,
	as in the first line of \eqref{NWKANEXEVEXT EQ}.
\end{proof}

Our last goal is to complete the proof of Theorem \ref{THMB}
by showing that, as claimed therein, 
the quadruples in Corollary \ref{NWXREPS COR}
always have a nice suitably unique representative.

We first discuss uniqueness of the maps
$\Omega[\mathfrak{n}] \xrightarrow{x} X$ up to degeneracy.

\begin{definition}\label{NECKDEG DEF}
	A map of necklaces 
	$(J\to T) \to (J' \to T')$
	is called a \emph{necklace degeneracy}
	if the associated map
	$\varphi \colon T \to T'$
	is a degeneracy in $\Omega$
	and $\varphi J = J'$.
\end{definition}

\begin{definition}[{cf. \cite[\S 4]{DS11}}]
        \label{TOTNONDEG_DEF}
	Let $J \xrightarrow{\mathfrak{n}} T$ be a necklace and 
	$X \in \mathsf{dSet}$.
	A map $\Omega[\mathfrak{n}] \to X$
	is called \emph{totally non-degenerate}
	if for all beads $T_b$, 
	$b \in \boldsymbol{B}(\mathfrak{n})$
	the induced dendrex
	$\Omega[T_b] \to X$
	is non-degenerate (in the sense of, 
	e.g. \cite[Prop. 5.62]{Per18}).
\end{definition}

\begin{lemma}[{cf. \cite[Prop. 4.7]{DS11}}]
	\label{DEGNECK LEM}
	Any map 
	$\Omega[\mathfrak{n}] \to X$
	has a factorization, unique up to unique isomorphism, as
\[
	\Omega[J \xrightarrow{\mathfrak{n}} T] \to 
	\Omega[J' \xrightarrow{\mathfrak{n'}} T'] \to X
\]
	where the first map is a degeneracy of necklaces
	and the second map is totally non-degenerate.
\end{lemma}

\begin{proof}
	The proof is by induction on the size of 
	$\boldsymbol{E}^{\mathsf{i}}(J)$.
	The base case is that of $\boldsymbol{E}^{\mathsf{i}}(J) = \emptyset$
	(note that then it must also be $\boldsymbol{E}^{\mathsf{i}}(J') = \emptyset$),
	in which case the result reduces to 
	\cite[Prop. 6.9]{CM11} or
	\cite[Prop. 5.62]{Per18}.

	Otherwise, let $e \in \boldsymbol{E}^{\mathsf{i}}(J)$
	and consider the grafting decomposition
	$T = R \amalg_e S$.
	By the induction hypothesis, 
	one has factorizations, unique up to unique isomorphism,
	$\Omega[\mathfrak{n}_{R}] \to \Omega[\mathfrak{n}_R'] \to X$,
	$\Omega[\mathfrak{n}_S] \to \Omega[\mathfrak{n}_{S}'] \to X$.
	Writing 
	$\mathfrak{n}_R' = (J'_R \to R')$
	and 
	$\mathfrak{n}_S' = (J'_S \to S')$,
	we then set
	$\mathfrak{n}'= \left(
	J'_R \amalg_e J'_S \to T'_R \amalg_e T'_S
	\right)$.
	The uniqueness up to unique isomorphism property of
	$\mathfrak{n}'$ is readily seen to be inherited from 
	the analogue property for 
	$\mathfrak{n}'_R,\mathfrak{n}'_S$
	(note that the ``unique isomorphism'' clause
	implies that there is no ambiguity concerning the grafting edge $e$),
	finishing the proof.
\end{proof}

Next, we also need a preferred form for the tall simplex data in \eqref{STRINCOR EQ}.

\begin{definition}[{cf. \cite[\S 4]{DS11}}]
        \label{FLANKED_DEF}
	A tall simplex as in 
	\eqref{STRINCOR EQ}
	is called \emph{flanked}
	if $J_0 = J$
	and $J_n = T$.
	Further, a quadruple
	$(\mathfrak{n},\phi,x,J_{\bullet})$
	is called \emph{flanked} if $J_{\bullet}$ is.
\end{definition}

\begin{remark}
	Suppose $(\mathfrak{n},\phi,x,J_{\bullet})$
	is a flanked quadruple
	and set $\mathfrak{n}_k = (J_k \to T)$.
	Then the structure maps in \eqref{STRINCOR EQ}
	induce a diagram of maps of necklaces
\begin{equation}
\begin{tikzcd}[column sep=12pt]
	\mathsf{Sc}[T]
	\ar[equal]{r} 
&
	\Omega[\mathfrak{n}_n]
	\ar{r}
&
	\Omega[\mathfrak{n}_{n-1}]
	\ar{r}
&
	\cdots
	\ar{r}
&
	\Omega[\mathfrak{n}_0]
	\ar[equal]{r}
&
	\Omega[\mathfrak{n}]
&
	\mathsf{Sc}[S]
	\ar{l}
\end{tikzcd}
\end{equation} 
\end{remark}

\begin{remark}\label{FLNKNECDEG REM}
	If both simplices 
	$J_{\bullet},J'_{\bullet}$ in a pushforward diagram
	\eqref{NWTTSTAR EQ}
	are flanked,
	then the associated map of necklaces
	$\mathfrak{n} \to \mathfrak{n}'$
	is a degeneracy.
\end{remark}

In what follows we say a quadruple
$(\mathfrak{n},\phi,x,J_{\bullet})$
as in Corollary \ref{NWXREPS COR}
is \emph{flanked} if $J_{\bullet}$ is and
\emph{totally non-degenerate} if $x$ is.

\begin{lemma}[{cf. \cite[Lemma 4.5]{DS11}}]
	\label{FLANKING LEM}
	\begin{enumerate}[label=(\roman*)]
		\item Any quadruple $(\mathfrak{n},\phi,x,J_{\bullet})$
		as in Corollary \ref{NWXREPS COR} is equivalent a flanked one;
		\item if two flanked quadruples are equivalent, then the equivalence can be described via a zigzag involving only flanked quadruples.
	\end{enumerate}
\end{lemma}

\begin{proof}
	The key to (i) is the fact that 
	the map $J_n \to T$
	induces a map of necklaces
	$(J_0 \to J_n) \to (J \to T)$.
	This map of necklaces induces a pushforward of simplices
	(i.e. a diagram as in \eqref{NWTTSTAR EQ})
\begin{equation}\label{FLANKING EQ}
\begin{tikzcd}
	S \ar{r}{t} \ar[equal]{d}
&
	J_0 \ar{r}{i,p} \ar[equal]{d}
&
	J_1 \ar{r}{i,p} \ar[equal]{d}
&
	\cdots
	\ar{r}{i,p}
&
	J_n \ar[equal]{r} \ar[equal]{d}
&
	J_n \ar{d}
\\
	S \ar{r}{t}
&
	J_0 \ar{r}{i,p}
&
	J_1 \ar{r}{i,p}
&
	\cdots
	\ar{r}{i,p}
&
	J_n \ar{r}{i,p}
&
	T
\end{tikzcd}
\end{equation}	
	where the top simplex (and thus the associated quadruple) is now flanked, so (i) follows.

	(ii) then follows by noting that the procedure above is natural.
	More precisely, 
	an arbitrary pushforward of tall simplices
	(i.e. simplices whose composite map is a tall map)
	along the necklace map $(J,T) \to (J^{\**} \to T^{\**})$
	as in \eqref{NWTTSTAR EQ}
	induces a pushforward of flanked simplices
\begin{equation}
\begin{tikzcd}
	S \ar{r}{t} \ar[equal]{d}
&
	J_0 \ar{r}{i,p} \ar{d}{d}
&
	J_1 \ar{r}{i,p} \ar{d}{d}
&
	\cdots \ar{r}{i,p}
&
	J_n \ar[equal]{r} \ar{d}{d}
&
	J_n \ar{d}
\\
	S \ar{r}{t} 
&
	J'_0 \ar{r}{i,p}
&
	J'_1 \ar{r}{i,p}
&
	\cdots \ar{r}{i,p}
&
	J'_n \ar{r}{f,p}
&
	T'
\end{tikzcd}
\end{equation}
along the necklace map
$(J_0 \to J_n) \to (J'_0 \to J'_n)$.
\end{proof}

\begin{corollary}[{cf. \cite[Cor. 4.8]{DS11}}]
        \label{NWXREPS2_COR}
	Each quadruple $(\mathfrak{n},\phi,x,J_{\bullet})$ as in Corollary \ref{NWXREPS COR}
	has a representative, unique up to unique isomorphism,
	which is both flanked and totally non-degenerate.
\end{corollary}

\begin{proof}
	By Lemma \ref{FLANKING LEM}(i)
	any quadruple is equivalent to a flanked quadruple 
	and, by Lemma \ref{DEGNECK LEM},
	any flanked quadruple is equivalent to a flanked quadruple that is also totally non-degenerate.
	
	As for the uniqueness condition, 
	by Lemma \ref{FLANKING LEM}(ii)
	we need only consider zigzags of 
	equivalences of flanked quadruples, 
	which are induced by necklace degeneracies, 
	in the sense of Definition \ref{NECKDEG DEF},
	cf. Remark \ref{FLNKNECDEG REM}.
	Thus, arguing by induction on the size of the zigzag,
	Lemma \ref{DEGNECK LEM} implies that 
	all flanked quadruples in the zigzag have the same
	totally non-degenerate quotient,
	so the desired uniqueness claim reduces to the uniqueness claim
	in Lemma \ref{DEGNECK LEM}.
\end{proof}

We conclude the paper by using Theorem \ref{THMB}
to describe $W_!$ applied to the 
key dendroidal sets in \S \ref{DSETSOPS SEC}.
We first make some useful observations
concerning $W_!$ applied to a representable $\Omega[U]$.

\begin{example}\label{WREP_EX}
	For $X=\Omega[U]\in \mathsf{dSet}$,
	one can describe $W_!(\Omega[U])$
	via either Proposition \ref{PROPA PROP} or Theorem \ref{THMB}.
	In preparation for the next examples, 
	which require Theorem \ref{THMB},
	we will find it useful to work out how 
	Theorem \ref{THMB} recovers Proposition \ref{PROPA PROP}. 
	Putting together all the data in the unique representative description
	in Theorem \ref{THMB},
	a simplex of $W_!(\Omega[U])$ is strictly uniquely represented by
\begin{equation}\label{WOMT EQ}
\begin{tikzcd}
	S \ar{r}{t} &
	J_0 = J \ar{r}{i,p} &
	J_1 \ar{r}{i,p} &
	\cdots \ar{r}{i,p} &
	J_n = T \ar{r}{f,p}[swap]{\phi} &
	U.
\end{tikzcd}
\end{equation}
This requires some justification. 
First, note that the role of $\phi$ is to represent a map
$\phi \colon \Omega[J\to T] \to \Omega[U]$. 
Then, the requirement in Theorem \ref{THMB} that $\phi$ is totally non-degenerate 
as a map of necklaces reduces to the implied claim in 
\eqref{WOMT EQ} that $\phi$ is a face map of trees.
The conditions $J_0 =J$ and $J_n = T$ are the flanked conditions.
Lastly, the assumption in \eqref{WOMT EQ}
that $\phi$ is planar is a \emph{choice}, 
which one is free to make, which turns the 
``uniqueness up to unique isomorphism''
in Theorem \ref{THMB} into strict uniqueness.
We now see that \eqref{WOMT EQ} indeed recovers \eqref{PROPA EQ}
(that this is also compatible with the simplicial structure follows from the flanking procedure in \eqref{FLANKING EQ}).

We now apply Remark \ref{READMAPSP REM}
to determine the mapping spaces of $W_!(\Omega[T])$.
Since this operad has color set 
$\boldsymbol{E}(U)$,
we consider signatures 
$(e_1,\cdots,e_n;e_0)$ with 
$e_i \in \boldsymbol{E}(U)$,
which can be regarded as a map
$e(-) \colon \boldsymbol{E}(C) \to \boldsymbol{E}(U)$
for $C$ the $n$-corolla.
We now claim it is 
(cf. \cite[\S 4]{CM13b})
\begin{equation}\label{WU_EQ2}
	W_!(\Omega[U])(e_1,\cdots,e_n;e_0) =
	\begin{cases}
	\Delta[1]^{\times \boldsymbol{E}^{\mathsf{i}}(\overline{e(C)})}
	\qquad
&
	\mbox{if $\boldsymbol{E}(C) \xrightarrow{e(-)} \boldsymbol{E}(U)$ defines a map in $\Omega$}
\\
	\varnothing
&
	\mbox{otherwise}
\end{cases}
\end{equation}
where $\overline{e(C)}$ is the outer closure notation in 
Notation \ref{TREEFACT NOT}.
In words, $\overline{e(C)}$ is the unique outer face of $U$
with leaves $e_1,\cdots,e_n$ and root $e_0$,
if such tree exists.
The identification \eqref{WU_EQ2}
now follows by setting $S=C$ in \eqref{WOMT EQ}. 
Indeed, if $e(-)$ does not define a map $C \to U$,
then no factorizations as in \eqref{WOMT EQ} exist.
And, otherwise, the only restriction on the $J_i$ therein
is that they must be inner faces of $\overline{e(C)}$.
But then \eqref{WOMT EQ} computes the nerve of the poset
$\mathsf{Face}_{inn}(\overline{e(C)})$ of inner faces of $\overline{e(C)}$,
which coincides with the poset 
$(0 \to 1)^{\times \boldsymbol{E}^{\mathsf{i}}(\overline{e(C)})}$
of subsets of its inner edges,
establishing \eqref{WU_EQ2}.
\end{example}

\begin{example}\label{WPARTIALT_EX}
	We now apply Theorem \ref{THMB} to compute
	$W_!(\partial \Omega[U])$.
	By the discussion in Example \ref{WREP_EX},
	its simplices are uniquely represented just as in 
	\eqref{WOMT EQ},
	except with the caveat that $\phi$
	now represents a map
	$\phi \colon \Omega[J\to T] \to \partial \Omega[U]$.
	This imposes the following restriction: $U$ itself can not be a bead of the necklace $J\to T$,
	which amounts to either 
	$J\neq \mathsf{lr}(U)$ or $T\neq U$.
	
	As such, for any signature
	$(e_1,\cdots,e_n;e_0)$ of $\boldsymbol{E}(U)$
	which is not the left-root signature,
	one has
	$W_!(\partial \Omega[U])(e_1,\cdots,e_n;e_0) =
	W_!(\Omega[U])(e_1,\cdots,e_n;e_0)$,
	since then $T$ is a proper face of $U$.
	
	And, for the leaf-root signature $(\underline{l};r)$,
	this restriction amounts to excluding the boundary
	of the nerve of the poset 
	$\mathsf{Face}_{inn}(U)
	\simeq 
	(0 \to 1)^{\times \boldsymbol{E}^{\mathsf{i}}(U)}$,
	thus identifying 	$W_!(\partial \Omega[U])(\underline{l},r)$ 
	with the \emph{domain} of the iterated pushout product
	\[
	\left(
	\{0,1\} \to \Delta[1]
	\right)^{\square \boldsymbol{E}^{\mathsf{i}}(U)}.
	\]
\end{example}

\begin{example}\label{WPARTIALT2_EX}
	Let $U\in \Omega$ and 
	$\emptyset \neq E \subseteq \boldsymbol{E}^{\mathsf{i}}(U)$,
	and consider $W_!(\Lambda^E[U])$.
	As in Example \ref{WPARTIALT_EX},
	one now requires for $\phi$ in \eqref{WOMT EQ}
	to encode a map
	$\Omega[J \to V] \to \Lambda^E[U]$,
	which imposes the restriction that
	either $T \not \supseteq U-E$ or
	$J \neq \mathsf{lr}(U)$. 
		
	As in Example \ref{WPARTIALT_EX}
	one has 
	$W_!(\Lambda^E[U])(e_1,\cdots,e_n;e_0) = W(\Omega[U])(e_1,\cdots,e_n;e_0)$
	whenever $(e_1,\cdots,e_n;e_0) \not \simeq (\underline{l};r)$,
	as then $T$ can contain no inner faces.

	Lastly, for the leaf-root signature $(\underline{l},r)$,
	the given restrictions identify
	$W_!(\Lambda^E[U])(\underline{l},r)$ 
	with the \emph{domain} of the iterated pushout product
\[
	\left(
	\{0,1\} \to \Delta[1]
	\right)^{\square \boldsymbol{E}^{\mathsf{i}}(U)-E}
	\square
	\left(
	\{1\} \to \Delta[1]
	\right)^{\square E}.
\]
\end{example}

\providecommand{\bysame}{\leavevmode\hbox to3em{\hrulefill}\thinspace}
\providecommand{\MR}{\relax\ifhmode\unskip\space\fi MR }
\providecommand{\MRhref}[2]{%
	\href{http://www.ams.org/mathscinet-getitem?mr=#1}{#2}
}
\providecommand{\doi}[1]{%
	doi:\href{https://dx.doi.org/#1}{#1}}
\providecommand{\arxiv}[1]{%
	arXiv:\href{https://arxiv.org/abs/#1}{#1}}
\providecommand{\href}[2]{#2}

\makeatletter\@input{xxOC.tex}\makeatother

\makeatletter\@input{xxAC.tex}\makeatother

\makeatletter\@input{xxTAS.tex}\makeatother


\begin{thebibliography}{MW07}
	
	\bibitem[BPa]{BP_edss}
	P.~Bonventre and L.~A. Pereira, \emph{Equivariant dendroidal {S}egal spaces and
		{$G$-$\infty$-}operads}, to appear in \textit{Algebraic \& Geometric
		Topology}. arXiv preprint
	\href{https://arxiv.org/abs/1801.02110v3}{1801.02110v3}.
	
	\bibitem[BPb]{BP_TAS}
	\bysame, \emph{Equivariant dendroidal sets and simplicial operads}, arXiv
	preprint \href{https://arxiv.org/abs/1911.06399}{1911.06399}.
	
	\bibitem[BPc]{BP_geo}
	\bysame, \emph{Genuine equivariant operads}, to appear in \textit{Advances in Mathematics}. arXiv preprint
	\href{https://arxiv.org/abs/1707.02226v2}{1707.02226v2}.
	
	\bibitem[BPd]{BP_FCOP}
	\bysame, \emph{Homotopy theory of equivariant operads with fixed colors}, arXiv
	preprint \href{https://arxiv.org/abs/1908.05440v2}{1908.05440v2}.
	
	\bibitem[CM11]{CM11}
	D.-C. Cisinski and I.~Moerdijk, \emph{Dendroidal sets as models for homotopy
		operads}, J. Topol. \textbf{4} (2011), no.~2, 257--299.
	\doi{10.1112/jtopol/jtq039}
	
	\bibitem[CM13]{CM13b}
	\bysame, \emph{Dendroidal sets and simplicial operads}, J. Topol. \textbf{6}
	(2013), no.~3, 705--756. \doi{10.1112/jtopol/jtt006}
	
	\bibitem[DS11]{DS11}
	D.~Dugger and D.~I. Spivak, \emph{Rigidification of quasi-categories}, Algebr.
	Geom. Topol. \textbf{11} (2011), no.~1, 225--261.
	\doi{10.2140/agt.2011.11.225}
	
	\bibitem[MW09]{MW09}
	I.~Moerdijk and I.~Weiss, \emph{On inner {K}an complexes in the category of
		dendroidal sets}, Adv. Math. \textbf{221} (2009), no.~2, 343--389.
	\doi{10.1016/j.aim.2008.12.015}
	
	\bibitem[MW07]{MW07}
	I.~Moerdijk and I.~Weiss, \emph{Dendroidal sets}, Algebr. Geom. Topol.
	\textbf{7} (2007), 1441--1470. \doi{10.2140/agt.2007.7.1441}
	
	\bibitem[Per18]{Per18}
	L.~A. Pereira, \emph{Equivariant dendroidal sets}, Algebr. Geom. Topol.
	\textbf{18} (2018), no.~4, 2179--2244. \doi{10.2140/agt.2018.18.2179}
	
	\bibitem[Wei12]{Wei12}
	I.~Weiss, \emph{Broad posets, trees, and the dendroidal category}, arXiv
	preprint \href{https://arxiv.org/abs/1201.3987}{\texttt{1201.3987}}, 2012.
	
\end{thebibliography}
\end{document}